\documentclass{amsart}
\usepackage{epsf,graphicx,amsmath,verbatim,epsfig,caption}
 \numberwithin{equation}{section}
\numberwithin{figure}{section}
\numberwithin{table}{section}
\usepackage{multicol}
\usepackage{booktabs}

\usepackage{chngpage}
\usepackage{enumitem}
\usepackage{float}
\usepackage{lipsum}

\newtheorem{theorem}{\sc Theorem}[section]
\newcommand{\RR}{{\rm I\kern -1.6pt{\rm R}}}

\makeatletter
\g@addto@macro{\endabstract}{\@setabstract}
\newcommand{\authorfootnotes}{\renewcommand\thefootnote{\@fnsymbol\c@footnote}}
\makeatother

\begin{document}
\begin{center}
\LARGE
\title{USING PEANO CURVES TO CONSTRUCT LAPLACIANS ON FRACTALS}
\maketitle

\normalsize
\authorfootnotes

DENALI MOLITOR \footnote{Research supported by the National Science Foundations through the Research Experiences for Undergraduates Program at Cornell University, grant DMS-1156350}\textsuperscript{1}, NADIA OTT \footnote{Research supported by the National Science Foundations through the Research Experiences for Undergraduates Program at Cornell University, grant DMS-1156350}\textsuperscript{2}, and ROBERT STRICHARTZ \footnote{Research supported in part by the National Science Foundation, grant DMS-1162045}\textsuperscript{3} \par \bigskip

\textsuperscript{1}\emph{Department of Mathematics, Colorado College} \par
\emph{902 N Cascade Ave, Colorado Springs, CO, 80903} \par
\emph{\email{Denali.Molitor@ColoradoCollege.edu}} \par \bigskip

\textsuperscript{2}\emph{Department of Mathematics, San Diego State University} \par
\emph{5500 Campanile Dr., San Diego, CA, 92182} \par 
\emph{nadiaott@yahoo.com}\par \bigskip

\textsuperscript{3}\emph{Department of Mathematics, Cornell University} \par
\emph{563 Malott Hall, Ithaca, NY, 14850} \par
 \emph{str@math.cornell.edu} \par

\end{center}


%

\begin{abstract}

We describe a new method to construct Laplacians on fractals using a Peano curve from the circle onto the fractal, extending an idea that has been used in the case of certain Julia sets. The Peano curve allows us to visualize eigenfunctions of the Laplacian by graphing the pullback to the circle. We study in detail three fractals: the pentagasket, the octagasket and the magic carpet. We also use the method for two nonfractal self-similar sets, the torus and the equilateral triangle, obtaining appealing new visualizations of eigenfunctions on the triangle. In contrast to the many familiar pictures of approximations to standard Peano curves, that do no show self-intersections, our descriptions of approximations to the Peano curves have self-intersections that play a vital role in constructing graph approximations to the fractal with explicit graph Laplacians that give the fractal Laplacian in the limit. 
\smallskip

\noindent $\textbf{\emph{Keywords:}}$ Peano curve, self-similiar fractals, Laplacian, pentagasket, octagasket, magic carpet
\smallskip


\end{abstract}


\section{Introduction}
In classical analysis, Peano curves are usually thought of as curiosities. It is exciting and counterintuitive that they exist, but they are not useful for solving problems. The purpose of this paper is to show that in fractal analysis, in contrast, Peano curves may play an important role in constructing Laplacians and studying their properties. In fact, this approach has already been used to construct Laplacians on Julia sets of complex polynomials \cite{julia},\cite{famofjulia},\cite{JuliaSetsIII}.

If $X$ is a compact topological space, we will use the term Peano curve for any continuous mapping $\gamma$ from the circle (parameterized by $t\in [0,1]$ with $0\equiv 1$) onto $X$. It is well known that $\gamma$ cannot be one-to-one (except in the trivial case in which $X$ is homeomorphic to a circle), so there must be values in $[0,1]$ that are mapped to the same point in $X$, say $\gamma (t_1)=\gamma (t_2)$. We will say that such $t_1$ and $t_2$ are identified and write $t_1\equiv t_2$. If we consider all possible identifications, then we obtain a model of $X$ as a circle with identifications, exactly the point of view adopted in \cite{julia},\cite{famofjulia},\cite{JuliaSetsIII}.. Typically the number of identifications will be uncountable. Our goal is to obtain a countable sequence of identifications that is appropriately dense in the set of all identifications, so that if we take an initial segment of identifications, we will obtain a useful approximation for $X$.

In the case of Julia sets, the Peano curve is defined by means of applying the Riemann mapping theorem to the component of infinity of the complement of the Julia set. Despite this rather abstract construction, it turned out that it was possible to find a useful sequence of identified points. In this paper, Peano curves will be given as a limit of a sequence of curves defined by a self-similar type iteration scheme, and the self-intersections of the approximating curves will give us the sequence of identifications of the limiting Peano curve. It is interesting to note that many of the examples of Peano curves to planar domains that are found in textbooks are also constructed as limits of simpler curves obtained by iterating some self-similar scheme, but the approximating curves have no self-intersections. From our point of view this is a rather silly choice.

Given a finite set of identifications on the circle, we have a natural graph structure where the identified points are the vertices, and the edges join consecutive points around the circle. If we assign positive weights $\mu(t_j)$ to the points, thought of as a discrete measure on the set of vertices, and non-negative weights $c(t_j,t_{j+1})$ to the edges, thought of as conductances on an electrical network associated to the graph, then we may define a graph Laplacian

\begin{equation}
-\Delta u(x)=\frac{1}{\mu (x)}\sum_{x \sim y} c(x,y)(u(x)-u(y)),
\end{equation}

where the sum is taken over all $y$ neighboring $x$ \cite{str}. For simplicity we may take $c(x,y)=0$ if $x$ and $y$ are not connected by an edge. Notice that self-edges are possible if $x$ and $y$ are identified neighboring points, but they do not contribute to the Laplacian. We want the discrete measures on the approximating graphs to converge to some natural measure on $X$. In all the examples considered here, we will use the simple choice of taking $\mu(x)=\sum\frac{1}{2}\left(t_{j+1}-t_{j-1}\right)$ where the sum is taken over all points in the set of identified points denoted by $x$. For the conductances, the simplest choice is $c\left(t_j,t_{j+1}\right)=\frac{1}{t_{j+1}-t_j}$, so the length of the interval $[t_j,t_{j+1}]$ is taken to be the resistance in the electrical network. This is not always the optimal choice, however, as is apparent in the previous work on Julia sets. In particular we will make a more sophisticated choice for one of our examples, the pentagasket. In order to obtain a Laplacian on $X$ in the limit it will be necessary to renormalize the sequence of Laplacians on the approximating graphs, and this becomes a highly nontrivial problem.

In this paper we will look at five examples of self-similar sets $X$ with carefully chosen Peano curves. Only three of them are fractal. The nonfractal sets are the equilateral triangle $T$ with Neumann boundary conditions, and the square torus $T_o$. These may be regarded as "controls," since we know exactly what the eigenvalues and eigenfunctions of the Laplacian are. We can then see how well the Peano curves reproduce the known results. It would of course be silly to claim that this is an optimal way to develop properties of the usual planar Laplacian; however we will show that it reveals some insights into the eigenfunctions to view them as functions on the circle pulled back via the Peano curves. 

The first fractal example we consider is the pentagasket, PG (Figure 1.1). It is a highly symmetric example of Kigami's class of postcritically finite (PCF) self-similar sets \cite{ki} \cite{str}. It has a fully symmetric self-similar Laplacian, whose properties were studied in detail in \cite{Spec}. Since it does not satisfy spectral decimation, many of these properties have only been observed experimentally. We construct a Peano curve that yields a slightly different sequence of graph approximations than was considered in \cite{Spec}. By following the outline from the work on Julia sets, we are able to find the correct renormalization constant and give an independent construction of the Laplacian, together with numerical results fully consistent with those in \cite{Spec}. In this sense, the case of PG may also be considered a "control."

\begin{figure}
\includegraphics[scale=0.5, trim=0mm 0mm 0mm 0mm,clip]{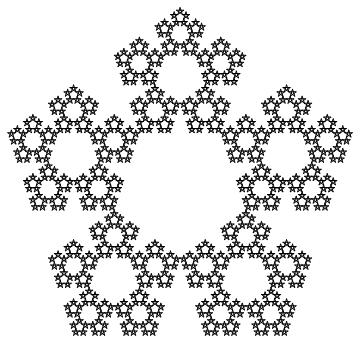}
\caption{Pentagasket}
\end{figure}

The second fractal example is the octagasket, OG (Figure 1.2). This fractal is not PCF, and in fact there is no proof yet of the existence of a symmetric self-similar Laplacian on OG. Nevertheless, experimental evidence of the existence of a Laplacian and properties of the spectrum using the method of "outer approximation" was presented in \cite{BHS}. Our Peano curve approach gives independent experimental evidence for the existence of the Laplacian with data that is consistent with \cite{BHS}. What is more, we are able to give concrete conjectures concerning the structure of the spectrum; for example, we find precise locations in the spectrum for spectral gaps that were noticed in \cite{BHS}. The existence of spectral gaps is still quite a mystery. It is possible to prove existence for PCF fractals that enjoy spectral decimation, but the proof is just technical and yields no ``ideological" explanation for them. The significance of spectral gaps is pointed out forcefully in \cite{FourierSeries}.

\begin{figure}[h!]
\includegraphics[scale=0.5, trim=0mm 0mm 0mm 0mm,clip]{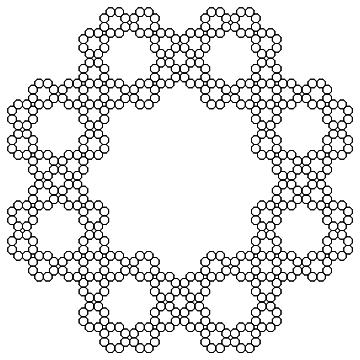}
\caption{Octagasket}
\end{figure}

The third fractal example we consider is the Magic carpet MC, recently introduced in \cite{kform}. This fractal is obtained by modifying the construction of the Sierpinski carpet SC (Figure 1.3) to immediately sew up all cuts that are made, as illustrated in Figure 1.4. Thus MC is a limit of closed surfaces, and geometrically these surfaces are flat everywhere with the exception of a finite number of point singularities carrying negative curvature. Again, there is no proof of the existence of a symmetric self-similar Laplacian on MC, but the results in this paper together with \cite{kform} give strong experimental evidence. The evidence suggests that analysis on MC should be similar to Euclidean spaces of dimension exceeding two, with points having zero capacity.

\begin{figure}
\includegraphics[scale=0.5, trim=0mm 0mm 0mm 0mm,clip]{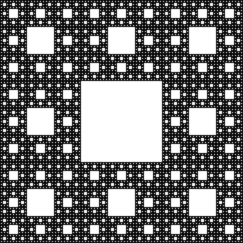}
\caption{Sierpinski Carpet}
\end{figure}

The paper is organized as follows. In section 2 we give descriptions of all the Peano curves, as limits of piecewise linear maps obtained recursively using specified substitution rules. We are not claiming that these are the unique curves, or even that they are optimal in any sense. 

In section 3 we discuss the example of PG. We describe the method of creating graph approximations and associated graph Laplacians from the Peano curve, and relate this to the Laplacian constructed by Kigami \cite{ki}(see also \cite{FourierSeries} and \cite{Spec}). We use the same method for the other examples, the octagasket (section 4), the magic carpet (section 5), and the torus and triangle (section 6). The existence of the Laplacian in the limit for OG and MC has not yet been established, and we do no see any way to use the Peano curve construction to resolve this problem. For the torus and the triangle the Laplacian and its spectrum are well known. The graphs of the eigenfunctions pulled back to the circle via the Peano curve for the triangle are new and appealing.

\begin{figure}
\begin{minipage}[h!]{0.40\linewidth}
\includegraphics[scale=.25,trim = 0mm 0mm 0mm 0mm, clip]{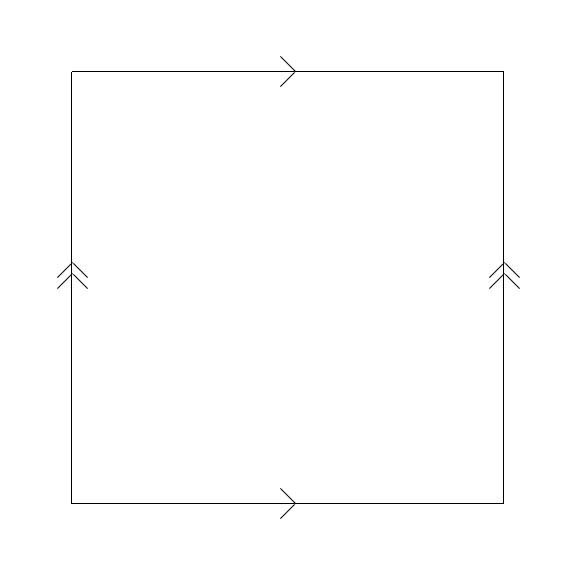}
\end{minipage}
\begin{minipage}[h!]{0.40\linewidth}
\includegraphics[scale=.25,trim = 0mm 0mm 0mm 0mm, clip]{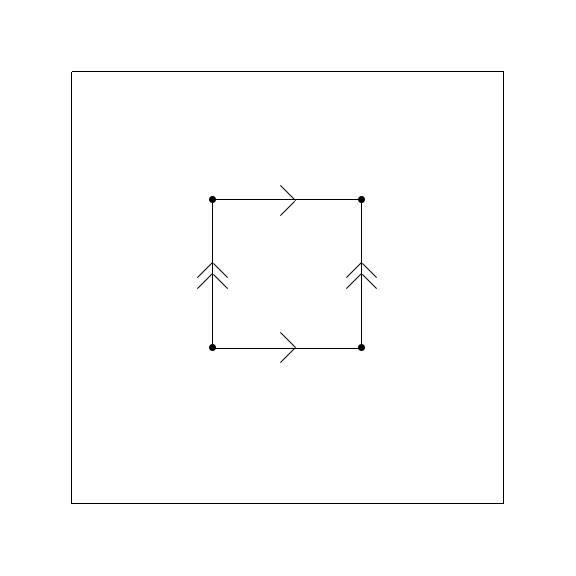}
\end{minipage}
\caption*{FIGURE 1.4. To transform SC into MC, "sew up" torus-type edge identifications, as shown here on level 0 and level 1. Point singularities of infinite negative curvature occur at the identified interior points markes at level 1}
\end{figure}

We have not attempted to obtain higher order accuracy in approximating eigenvalues by using more sophisticated numerical methods, such as the finite element method. It should not be difficult to adopt such methods to our examples, if so desired.

We include many numerical results, such as tables of eigenvalues and graphs of eigenfunctions. The reader should see the website \emph{www.math.cornell.edu/ $\sim$ dmolitor} for more data.

This paper should be regarded as a "proof of concept" paper, showing by example that the method of Peano curves can be an effective tool in studying analysis on fractals. We hope it will prove useful for other fractals.

\section{Constructing the Peano Curves}
All the Peano curves we consider are limits of piecewise linear maps $\gamma_m$, where the passage from $\gamma_m$ to $\gamma_{m+1}$ is given by a set of substitution rules for replacing each linear segment of $\gamma_m$ by a union of consecutive segments of $\gamma_{m+1}$ with the same endpoints. To make this clear we start by describing a Peano curve whose image is the Sierpinski gasket (SG, Figure 2.1).
\begin{figure}
\includegraphics[scale=.5,trim = 15mm 45mm 75mm 65mm, clip]{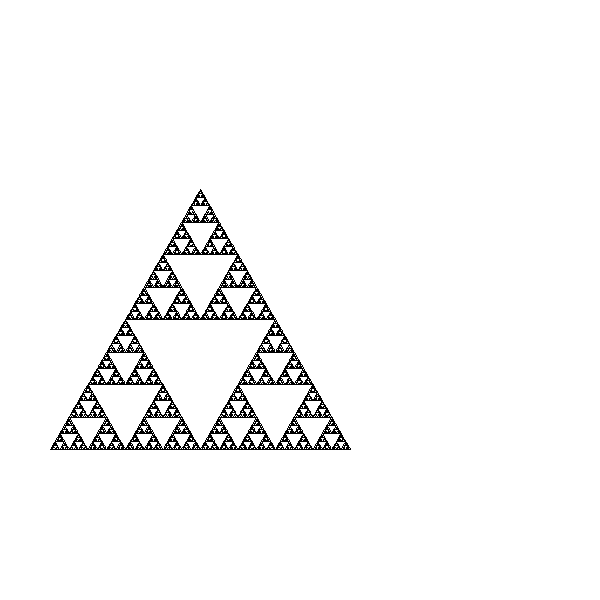}
\caption{Sierpinski Gasket.}
\end{figure}
 We will not discuss this example in detail below because it gives rise to exactly the same approximations to the Laplacian on SG as initially constructed by Kigami \cite{ki} \cite{str}. The first approximation $\gamma_0$ simply traces out an equilateral triangle at constant speed. The substitution rules are shown in Figure 2.2. Each linear segment traces out an interval in one of the three orientations of the sides of the triangle. Each segment of $\gamma_m$ is parameterized by an interval $\left[\frac{k}{3^{m+1}},\frac{k+1}{3^{m+1}}\right]$ and is shown as a dotted line in the figure. It's replacement, three intervals parameterized by $\left[\frac{3k}{3^{m+2}},\frac{3k+1}{3^{m+2}}\right],\ \left[\frac{3k+1}{3^{m+2}},\frac{3k+2}{3^{m+2}}\right],\ \left[\frac{3k+2}{3^{m+2}},\frac{3k+3}{3^{m+2}}\right]$ traces out the solid lines, with the same direction as the dotted line.
\begin{figure}
\begin{minipage}[h!]{0.28\linewidth}
\includegraphics[scale=.28,trim = 0mm 0mm 0mm 0mm, clip]{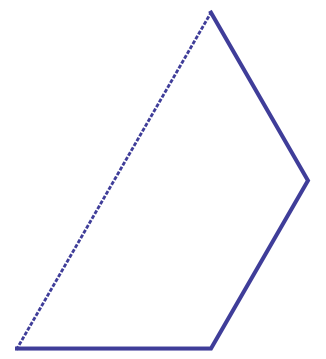}\\a.
\end{minipage}
\begin{minipage}[h!]{0.28\linewidth}
\includegraphics[scale=.28,trim = 0mm 0mm 0mm 0mm, clip]{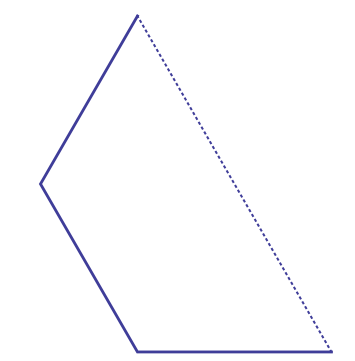}\\b.
\end{minipage}
\begin{minipage}[h!]{0.28\linewidth}
\includegraphics[scale=.28,trim = 0mm 0mm 0mm 0mm, clip]{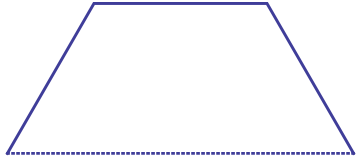}\\c.
\end{minipage}
\caption{Shows the substitution rule for the three types of line segments encountered in SG. The dotted line represents the line from the previous level that is to be replaced. The replacement curve has the same beginning and end point as the line in the previous level.}
\end{figure}
Figure 2.3 shows the image of $\gamma_0, \gamma_1$ and $\gamma_2$ with arrows to show the direction and vertices labeled $k$ to indicate $\gamma_m\left(\frac{k}{3^{m+1}}\right)$.
\begin{figure}
\begin{minipage}[h!]{0.32\linewidth}
\includegraphics[scale=.25,trim = 0mm 0mm 0mm 0mm, clip]{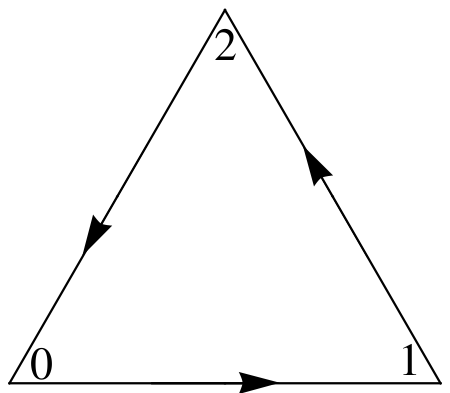}\\a.
\end{minipage}
\begin{minipage}[h!]{0.32\linewidth}
\includegraphics[scale=.25,trim = 0mm 0mm 0mm 0mm, clip]{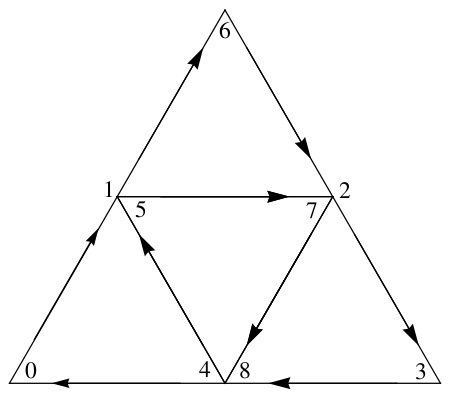}\\b.
\end{minipage}
\begin{minipage}[h!]{0.32\linewidth}
\includegraphics[scale=.25,trim = 0mm 0mm 0mm 0mm, clip]{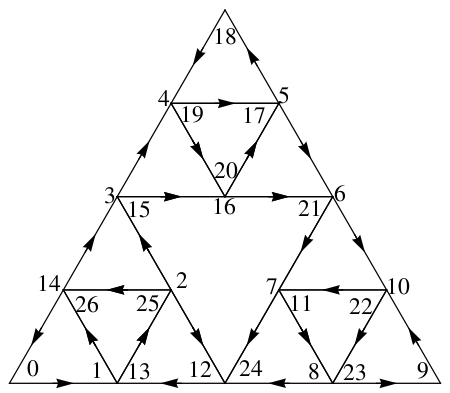}\\c.
\end{minipage}
\caption*{FIGURE 2.3. The image of $\gamma_m$ for $m=0,1,2$ in units of $(\frac{1}{3})^{m+1}$}
\end{figure}
Note that $\gamma_m\left(\frac{k}{3^{m+1}}\right)=\gamma_{m+1}\left(\frac{3k}{3^{m+2}}\right)$ etc., so the value of the limiting curve $\gamma$ at a value $\frac{k}{3^{m+1}}$ is the same as for all $\gamma_{m'}$ with $m'\geq m$. For a generic value of $t$, however, it is not so obvious what the point $\gamma (t)$ on SG is exactly. In Figure 2.4,
\begin{figure}
\begin{minipage}[h!]{0.49\linewidth}
\includegraphics[scale=.6,trim = 0mm 0mm 0mm 0mm, clip]{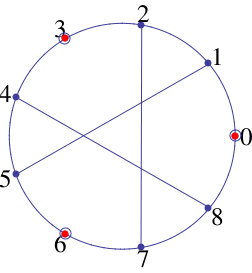}\\a.
\end{minipage}
\begin{minipage}[h!]{0.49\linewidth}
\includegraphics[scale=.45,trim = 0mm 0mm 0mm 0mm, clip]{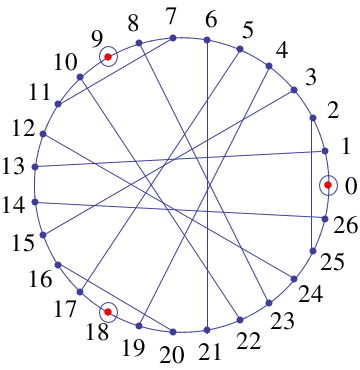}\\b.
\end{minipage}
\caption*{FIGURE 2.4. Shows the identifications on the Peano curves corresponding to the first and second level graph approximations of the Sierpinski Gasket. Figure 2.4.a corresponds to the first level graph approximation  in units of $\frac{1}{9}$ and figure 2.4b corresponds to the second in units of $\frac{1}{27}$. Note that for each level $m$, the points $0, 3^m$ and $2\cdot 3^m$ remain unidentified. Also, identifications from previous levels carry on as permanent fixtures in each of the higher levels.}
\end{figure}
 we show the identifications on the circle for the curves $\gamma_1$ and $\gamma_2$. These identifications persist for the limit curve $\gamma$ and all identifications arising from $\gamma$ are limits of those produced by $\gamma_m$ as $m\to\infty$, although it is not at all obvious what these are.

To indicate the complexity of this simple example, we ask the question: what is the image of $\gamma \left(\left[0,\frac{1}{3}\right]\right)$, the first third of the circle? Figure 2.5
\begin{figure}
\includegraphics[scale=.5,trim = 15mm 50mm 75mm 60mm, clip]{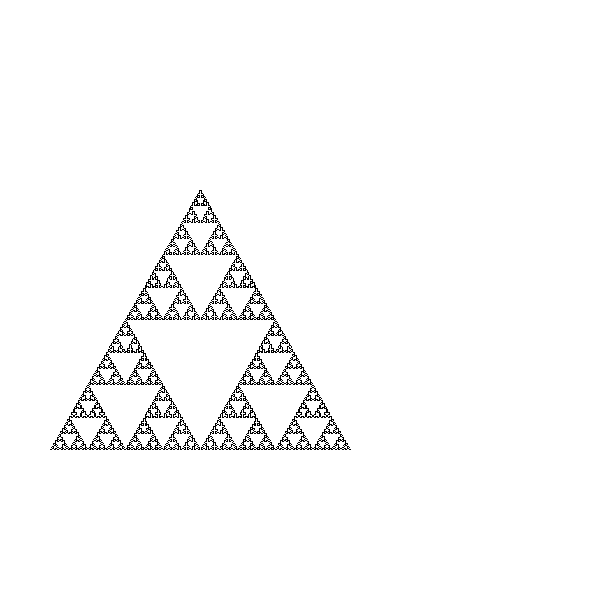}
\caption*{FIGURE 2.5. The image of $\gamma_5$ $[0,\frac{1}{3}]$.}
\end{figure}
 shows $\gamma_7 \left(\left[0,\frac{1}{3}\right]\right)$, which suggests $\left(\left[0,\frac{1}{3}\right]\right)=SG$. It is in fact easy to prove this, since $\gamma_m\left(\left[0,\frac{1}{3}\right]\right)$ satisfies a self-similar identity which uniquely characterizes SG. Thus the restriction of $\gamma$ to $\left(\left[0,\frac{1}{3}\right]\right)\ \left(\text{or }\left[\frac{1}{3},\frac{2}{3}\right] \text{ or }\left[\frac{2}{3},1\right]\right)$ is another Peano curve describing SG. Why did we not use this Peano curve to start with?
There are two reasons: 
1) the Peano curve destroys the dihedral-3 symmetry,
2) the approximations have no self-intersections.
(If we were interested in creating a Peano curve from $\mathbb{R}$ to an infinite blow up of SG we would use these curves as the building blocks.) But this observation has the simple consequence that the original Peano curve $\gamma$ must pass through each point (with the exception of the three vertices of the original triangle) at least three times, once for each third of the circle. Of course the smaller curves must have self-intersections as well, so there must certainly be some equivalence classes of identified points with four elements. Nevertheless, the approximating curves only produce equivalence classes with two elements. There are many interesting questions about the size of equivalence classes that we are not able to answer; are there infinite equivalence classes? If not, is there a maximal size? What is the size of a "typical" equivalence class? What is the complete equivalence class of the junction points $\frac{k}{3^{m+1}}$? Despite the vexing nature of these unanswered questions, the identifications from the approximating curves $\gamma_m$ give a perfect picture of the Laplacian on SG. In this case, what you do not know does not hurt you.

The next Peano curve we describe maps to the Pentagasket (PG). The first level approximation $\gamma_1$ traces out a five-sided star. We think of this as consisting of five line segments, but because of the intersections, each segment consists of three pieces. We choose to traverse the pieces at different speeds, faster on the longer end pieces and slower on the short middle piece, so that the arrival times will be consistent from level to level. In general, the identified points for $\gamma_m$ will be of the form $\frac{k}{5^{m+1}}$, where $k\equiv 1$ or $4\mod 5$. Figure 2.6 shows $\gamma_1$ and $\gamma_2$ $\left(\text{we write }k\text{ for }\frac{k}{5^{m+1}}\right)$.
\begin{figure}
\begin{minipage}[h!]{0.49\linewidth}
\includegraphics[scale=.4,trim = 0mm 0mm 0mm 0mm, clip]{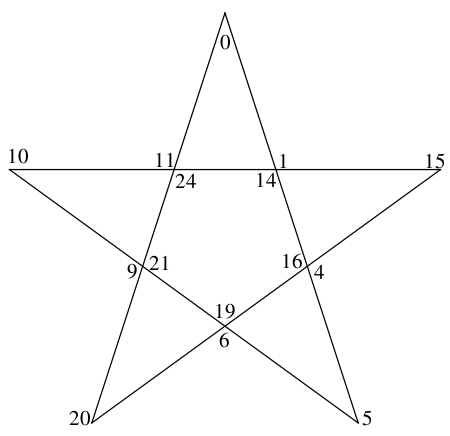}
\end{minipage}
\begin{minipage}[h!]{0.49\linewidth}
\includegraphics[scale=.4,trim = 0mm 0mm 0mm 0mm, clip]{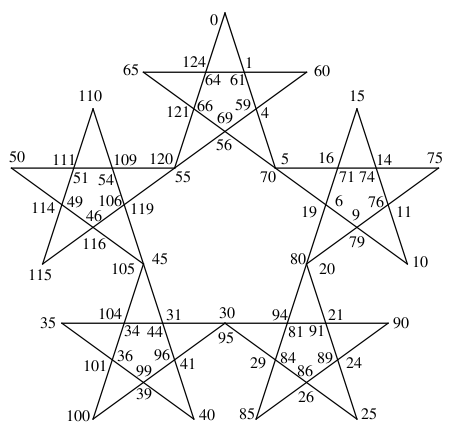}
\end{minipage}
\caption*{FIGURE 2.6. Shows the first and second level graph approximations of PG with the locations corresponding to locations on the Peano curve labeled on each of the vertices. In units of $\frac{1}{25}$ for the first level and $\frac{1}{125}$ for the second level.}
\end{figure}
The substitution rule is illustrated in Figure 2.7 and is the same for all five rotations of the original dotted line.

\begin{figure}
\includegraphics[scale=0.5,trim = 25mm 25mm 25mm 25mm, clip]{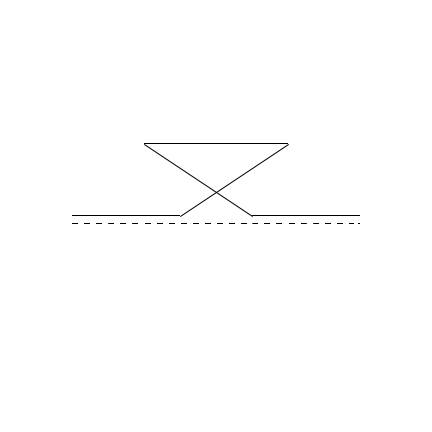}
\caption*{FIGURE 2.7. The substitution rule for one of the five directions.}
\end{figure}

\begin{figure}
\begin{minipage}[h!]{0.49\linewidth}
\includegraphics[scale=.4,trim = 0mm 0mm 0mm 0mm, clip]{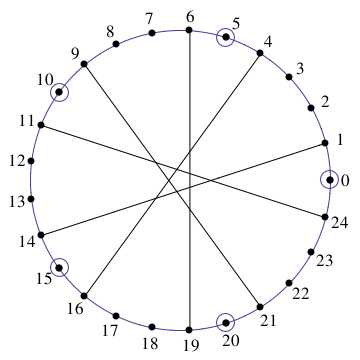}\\a.
\end{minipage}
\begin{minipage}[h!]{0.49\linewidth}
\includegraphics[scale=.4,trim = 0mm 0mm 0mm 0mm, clip]{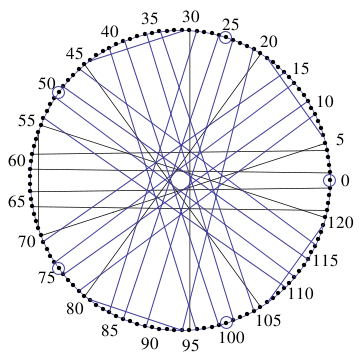}\\b.
\end{minipage}
\caption*{FIGURE 2.8. Shows the identifications on the Peano curve corresponding to the first and second level graph approximations of PG in units of $\frac{1}{25}$ and $\frac{1}{125}$, respectively. Note that the points of the form $5^{m-1}$, where $m$ is the level of the graph approximation remain unidentified throughout all levels of the graph approximations.}
\end{figure}
In Figure 2.8
we show the identifications on the circle arising from $\gamma_1$ and $\gamma_2$. We include the unidentified points of the form $\frac{k}{5^{m+1}}$, with $k\equiv 0\mod 5$ as these are turning points for the curves and we will use these in forming the approximating graphs and associated Laplacians. We see that there are two distinct types of edges in the graph of length $\frac{1}{5^{m+1}}$ and $\frac{3}{5^{m+1}}$ and we will assign different conductances to the two types. All equivalence classes of identified points consist of pairs and identifications persist from level to level. In this case we note that the image of one fifth of the circle is not all of PG, but rather a smaller self-similar set defined by an iterated function system of five similarities with a smaller contraction ration and a different set of fixed points (See Figure 2.9).
\begin{figure}
\includegraphics[scale=.3,trim = 0mm 0mm 0mm 0mm, clip]{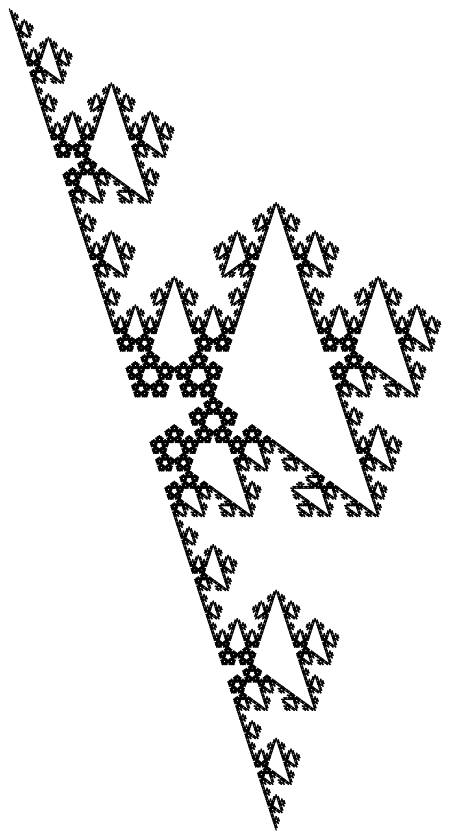}
\caption*{FIGURE 2.9. Shows the image of the first fifth of the Peano curve on PG. Each of the subsequent one-fifth sections of the Peano curve are reflections or rotations of this image.}
\end{figure}

Next we describe the Peano curve for OG. We begin with $\gamma_0$ that traces around an octagon clockwise and then turns around and traces around counterclockwise. The substitution rule is shown in Figure 2.10. 
\begin{figure}
\begin{minipage}[h!]{0.49\linewidth}
\includegraphics[scale=.45,trim = 10mm 15mm 30mm 10mm, clip]{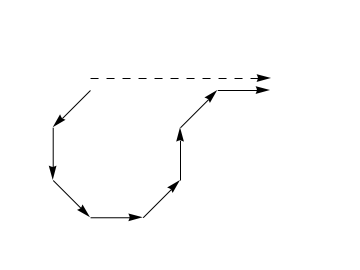}
\end{minipage}
\begin{minipage}[h!]{0.49\linewidth}
\includegraphics[scale=.45,trim = 10mm 15mm 10mm 10mm, clip]{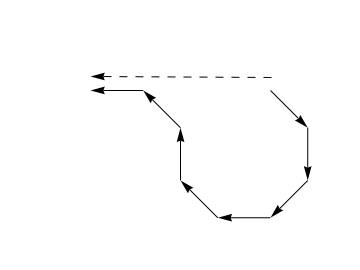}
\end{minipage}
\caption*{FIGURE 2.10. Shows the substitution rule for OG. The dotted arrow shows the line from the previous level of the graph approximation that is to be replaced. Note that the two substitutions shown differ in their direction and are reflections of each other. The eight rotations of these dotted lines corresponding to different possible edges of OG graph approximations use the same substitutions.}
\end{figure}
Note that the direction of the edge is significant. For $\gamma_m$ the identified points are of the form $\frac{k}{2\cdot 8^{m+1}}$. Some equivalence classes have two points, and we call these \emph{outer points}, while others have four points and we call these \emph{inner points}. Figure 2.11
\begin{figure}
\begin{minipage}[h!]{0.49\linewidth}
\includegraphics[scale=.45,trim = 10mm 15mm 30mm 10mm, clip]{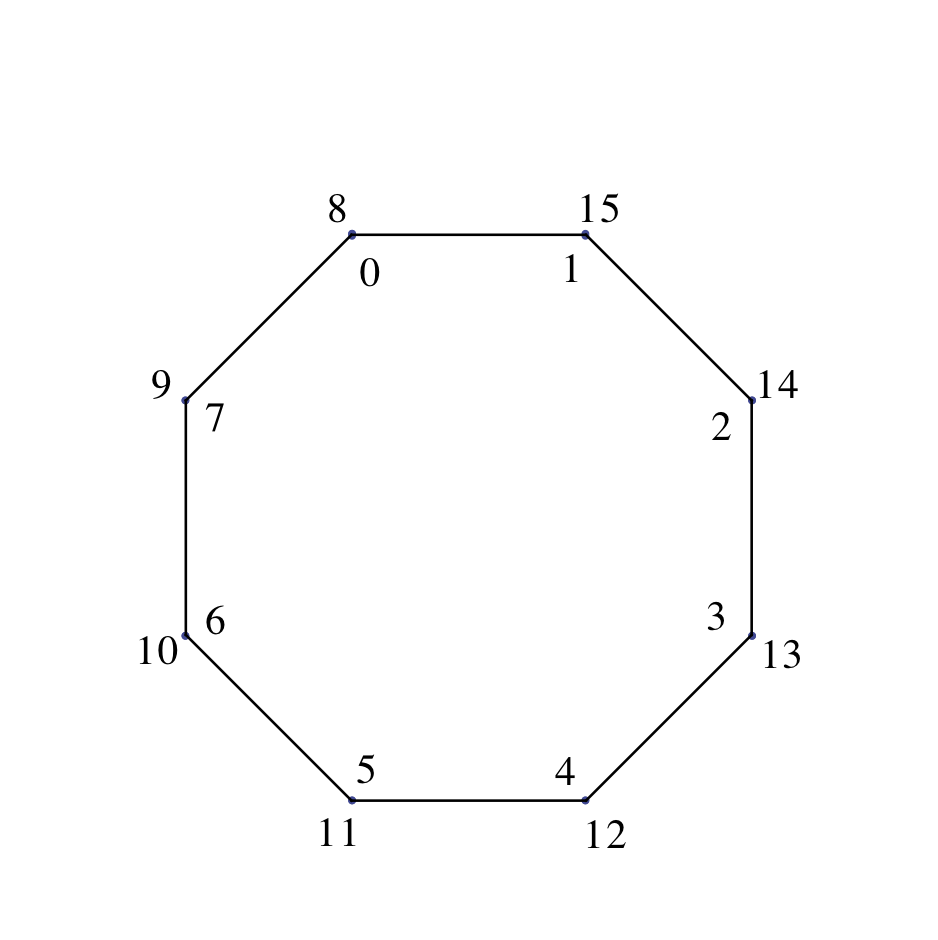}
\end{minipage}
\begin{minipage}[h!]{0.49\linewidth}
\includegraphics[scale=.35,trim = 10mm 15mm 10mm 10mm, clip]{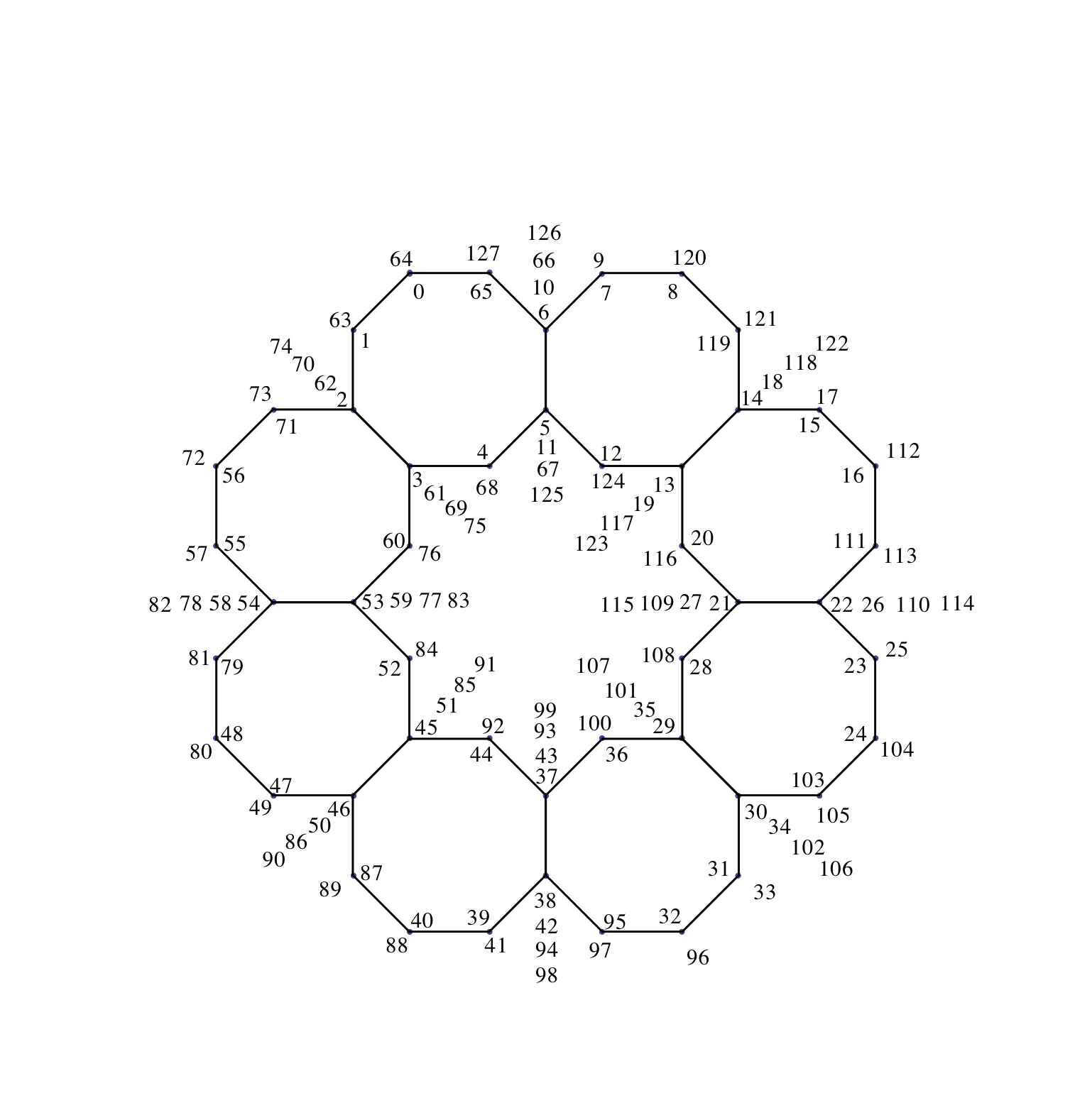}
\end{minipage}
\caption*{FIGURE 2.11. The paths $\gamma_1$ and $\gamma_2$ on OG}
\end{figure}
shows $\gamma_0$ and $\gamma_1$. Figure 2.12
\begin{figure}
\begin{minipage}[h!]{0.49\linewidth}
\includegraphics[scale=.40,trim = 0mm 0mm 0mm 0mm, clip]{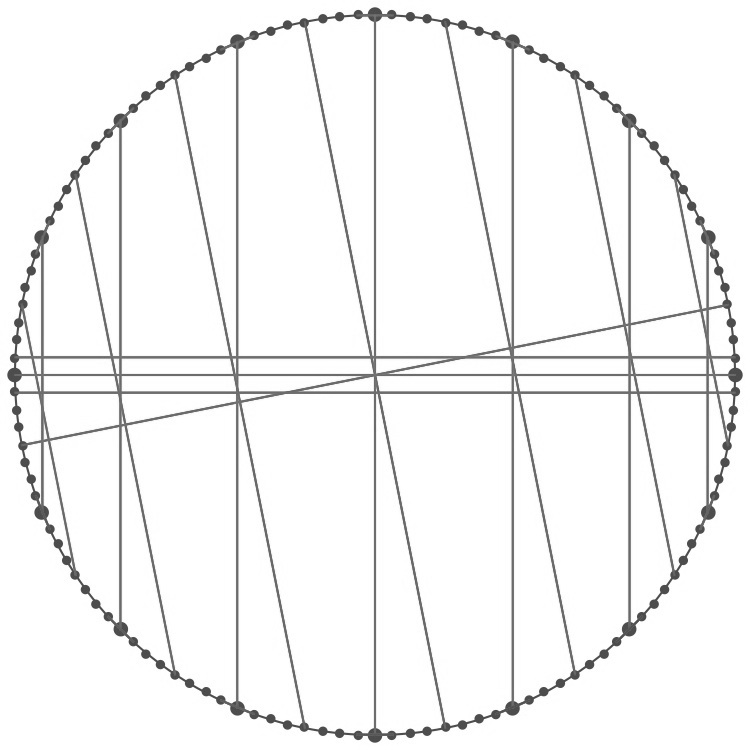}\\a.
\end{minipage}
\begin{minipage}[h!]{0.49\linewidth}
\includegraphics[scale=.40,trim = 0mm 0mm 0mm 0mm, clip]{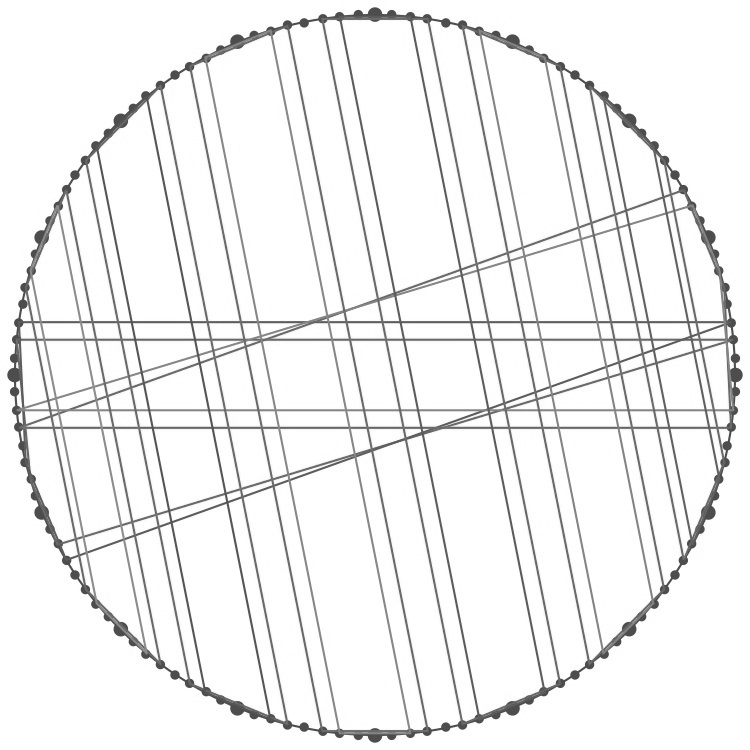}\\b.
\end{minipage}
\caption*{FIGURE 2.12. Shows the identifications of outer points(a) and inner points(b) from $\gamma_2$} 
\end{figure}
shows the identifications of outer and inner points from $\gamma_2$.

We note that the restriction of $\gamma$ to the interval $\left[ 0,\frac{1}{2}\right]$ maps onto OG, but it destroys the symmetry and fails to give the full set of identifications.

Next we describe the Peano for the Magic Carpet (MC). This construction requires a bit of imagination, since MC does not embed in the plane. The curves $\gamma_m$ approximating $\gamma$ may be visualized as mappings to the Sierpinski Carpet (SC) with jump discontinuities, where the jumps connect points of SC that are identified to create MC. We will take $\gamma_m$ to be the union of $2\cdot 8^{m}$ line segments from $\left[\frac{k}{2\cdot 8^{m+1}},\frac{k+1}{2\cdot 8^{m+1}}\right]$ to each of two boundary edges of each $m$-cell in MC.
In Figure 2.13
\begin{figure}
\begin{minipage}[h!]{0.49\linewidth}
\includegraphics[scale=.35,trim = 0mm 0mm 0mm 0mm, clip]{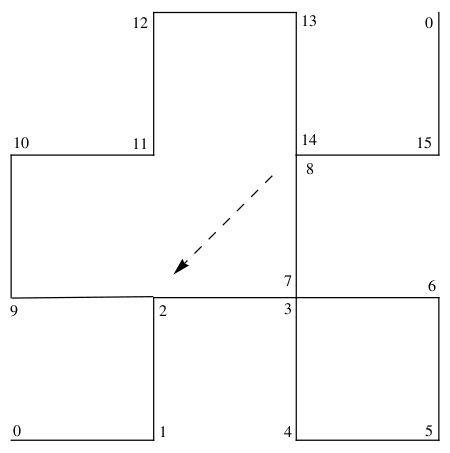}\\a.
\end{minipage}
\begin{minipage}[h!]{0.49\linewidth}
\includegraphics[scale=.35,trim = 0mm 0mm 0mm 0mm, clip]{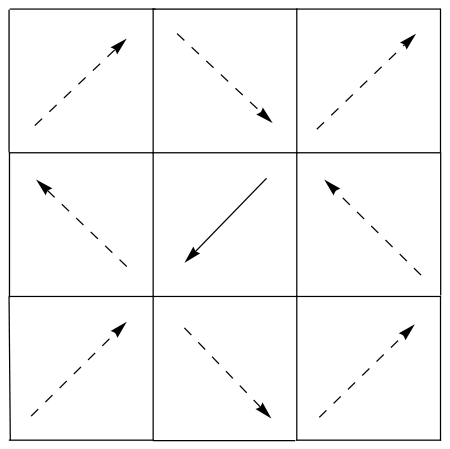}\\b.
\end{minipage}
\caption*{FIGURE 2.13: The path $\gamma_1$ for MC. In a. we show the actual path with the points 2,3,7,8,11 and 14 identified. In b., the dashed arrows show the symbolic description, while the solid arrow illustrates a jump between identified points.}
\end{figure}
we show both the symbolic description of $\gamma_1$ and the actual path, which has a single jump discontinuity at $\frac{1}{2}$. We note that this gives rise to a set of six identified points $\{2,3,7,8,11,14\}$ and the rest identified in pairs $\{0,5\},\{1,12\},\{4,13\},\{6,9\}$ and $\{10,15\}$, mainly because of the MC identifications. The substitution rule in symbolic form is shown in Figure 2.14.
\begin{figure}
\begin{minipage}[h!]{0.49\linewidth}
\includegraphics[scale=.6,trim = 0mm 0mm 0mm 0mm, clip]{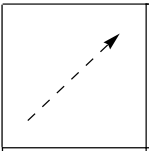}\\a.
\end{minipage}
\begin{minipage}[h!]{0.49\linewidth}
\includegraphics[scale=.35,trim = 0mm 0mm 0mm 0mm, clip]{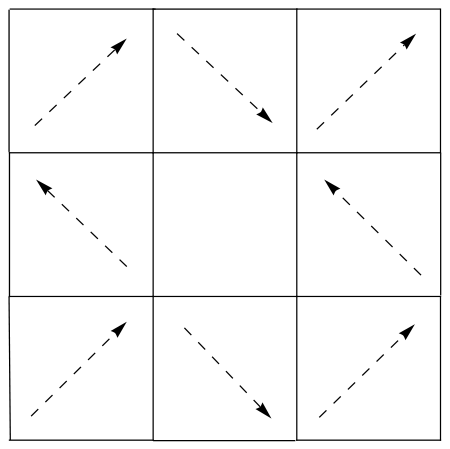}\\b.
\end{minipage}
\caption*{FIGURE 2.14. Shows one of  the symbolic substitution rule for MC.The others are simply rotations of this one.}
\end{figure}
\begin{figure}
\includegraphics[scale=.45,trim = 0mm 0mm 0mm 0mm, clip]{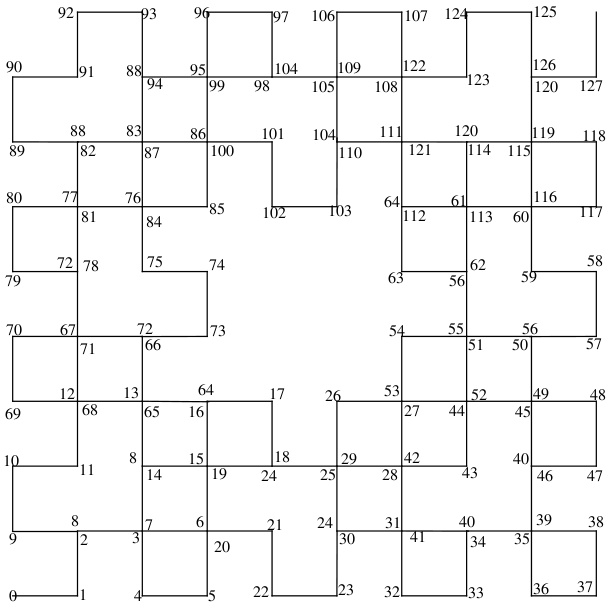}
\caption*{FIGURE 2.15. The path $\gamma_2$ on MC}
\end{figure}
Figure 2.15 shows the path $\gamma_2$. We note that every edge at level $m$ is traversed exactly once by $\gamma_m$.

This Peano curve differs from the others in that the identifications at level $m$ do not persist exactly at level $m+1$. In fact if $\frac{k}{2\cdot 8^{m+1}}\sim\frac{j}{2\cdot8^{m+1}}$ at level $m$ where $k$ is even and $j$ is odd, then $\frac{8k}{2\cdot8^{m+2}}\sim \frac{8j\pm 3}{2\cdot 8^{m+2}}$ at level $m+1$.

A slight modification of the MC construction gives rise to a Peano curve to the square torus T$_0$. Figure 2.16
shows how to modify the substitution rule from Figure 2.14. Figure 2.17 shows the actual path of $\gamma_1$.

\begin{figure}
\begin{minipage}{0.49\linewidth}
\includegraphics[scale=.6,trim = 0mm 0mm 0mm 0mm, clip]{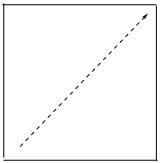}
\end{minipage}
\begin{minipage}{0.49\linewidth}
\includegraphics[scale=.6,trim = 0mm 0mm 0mm 0mm, clip]{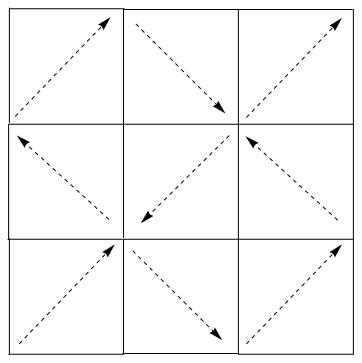}
\end{minipage}
\caption*{FIGURE 2.16. Shows one of the symbolic substitution rules for $T_0$. The others are simply rotations of this one.}
\end{figure}

\begin{figure}
\includegraphics[scale=.65,trim = 0mm 0mm 0mm 0mm, clip]{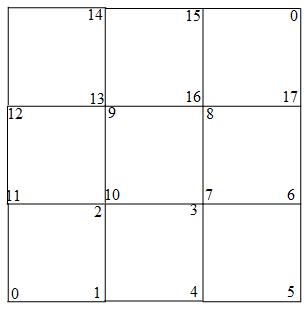}
\caption*{FIGURE 2.17. Path of $\gamma_1$ on $T_0$}
\end{figure}

Now $\gamma_m$ identifies points of the form $\frac{k}{2\cdot 9^{m}}$ in pairs. Again in this example we note that identifications do not persist from level to level. 

The final example gives a Peano curve to the equilateral triangle $T$. It may be thought of as a modification of the first example (SG). The substitution rule is shown in Figure 2.18 (to be compared with Figure 2.2), and the paths $\gamma_0,\gamma_1$ and $\gamma_2$ are shown in Figure 2.19 (to be compared with Figure 2.3). We note that there are now three types of points:

\begin{figure}
\centering
\begin{adjustwidth}{-.0in}{-0in}
\begin{minipage}[h!]{0.49\linewidth}
\includegraphics[scale=.4,trim = 5mm 40mm 5mm 5mm, clip]{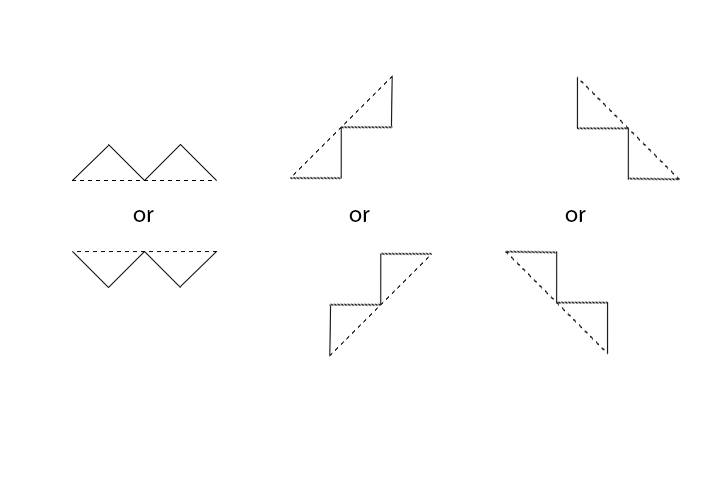}
\end{minipage}
\end{adjustwidth}
\caption*{Figure 2.18. Shows the substitution rule for the triangle}
\end{figure}

\begin{table}

\centering
\begin{adjustwidth}{-1.3in}{-1.0in}
\begin{tabular}{c c c}

\includegraphics[scale=.5,trim = 0mm 55mm 0mm 0mm, clip]{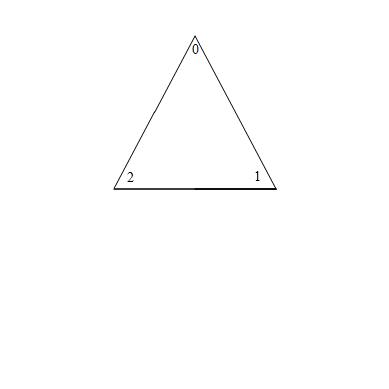}
&
\includegraphics[scale=.45,trim = 0mm 0mm 0mm 0mm, clip]{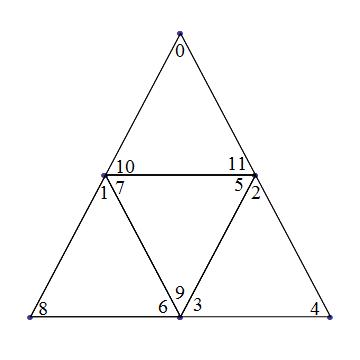}
&
\includegraphics[scale=.45,trim = 0mm 0mm 0mm 0mm, clip]{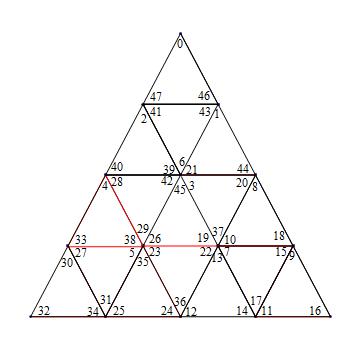}\\

\end{tabular}
\caption*{Figure 2.19. The paths $\gamma_0$, $\gamma_1$, and $\gamma_2$ for the triangle}
\end{adjustwidth}
\end{table}

(i) the three corner points that are not identified;
(ii) the other points along the boundary that are identified in groups of three;
(iii) interior points that are identified in groups of six. 
Similarly, there are two types of edges: (i) boundary edges that are traversed once in the counterclockwise direction;
(ii) interior edges that are traversed twice, both times in the same direction. For the boundary edges there is no choice of which of the substitution rules to use in order for the curve to stay inside the triangle. Since the interior edges bound two distinct cells, we need to make use of both alternatives, and our convention is to go into the central cell first and the peripheral cell the second time we traverse the edge. At level $m$, the identified points are of the form $\frac{k}{3\cdot 4^{m}}$, and in this example the identifications persist from level to level. The boundary of the triangle is the image of a Cantor set in the circle, as shown in Figure 2.20.

\begin{figure}
\includegraphics[scale=.4,trim = 0mm 0mm 0mm 0mm, clip]{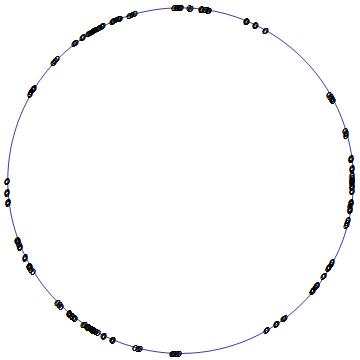}
\caption*{FIGURE 2.20. The Cantor set on the circle whose image under $\gamma$ corresponds to the boundary of $T$.}
\end{figure}

\section {The Pentagasket}
In this section we use the Peano curve $\gamma$ for the Pentagasket (PG) and its approximations $\gamma_m$ to define both an energy and a Laplacian that are self-similar and symmetric with respect to the dihedral-5 symmetry group that acts on PG. Since PG belongs to Kigami's PCF class, the energy and Laplacian are unique up to constant multiples, so our construction may be compared with previous constructions, in particular \cite{Spec}.

Let $\Gamma_m$ denote the graph determined by $\gamma_m$ as described in section two. We initially assign conductances to edges as follows:
\begin{equation}
c(x,y)=\begin{cases}
1\text{ if }[x,y]\text{ has length } \frac{1}{5^{m+1}}\\
b\text{ if }[x,y]\text{ has length } \frac{3}{5^{m+1}}
\end{cases},\end{equation}
where $b$ is a constant to be determined. Then set 
\begin{equation}
E_m(u)=\sum_{x\sim y} c(x,y)(u(x)-u(y))^2.
\end{equation}
Given a function $u$ on the vertices of $\Gamma_m$, we let $\bar u$ denote the extension of $u$ to the vertices of $\Gamma_{m+1}$ that minimizes the value of $E_{m+1}$, and then consider the renormalization equation
\begin{equation}
E_{m+1}(\bar u)=r E_m(u)
\end{equation}
for a renormalization factor $r$ to be determined. There is a unique choice of values for $b$ and $r$ for which (3.3) holds, namely
\begin{equation}
\begin{cases}
b=\frac{1+\sqrt{161}}{10} \\
r= \frac{\sqrt{161}-9}{8}
\end{cases}.
\end{equation}
Note that the value of $r$ agrees with the value given in \cite{Spec}.
Then we define the renormalized energy
\begin{equation}
\mathcal{E}_m (u)=r^{-m}E_m (u),
\end{equation}
making $\mathcal{E}_m (u)$ an increasing sequence, so 
\begin{equation}
\mathcal{E} (u)=\lim_{m\to\infty}\mathcal{E}_m (u).
\end{equation}
Here $u$ is any continuous function on the circle that respects all identifications made by $\gamma_m$ for all $m$, hence $u(t)=u(s)$ if $\gamma(t)=\gamma (s)$. We define $dom\, \mathcal{E}$ to be those functions with $\mathcal{E}(u)<\infty$. If $u,v\in dom\,\mathcal{E}$, then 
\begin{equation}
\mathcal{E}(u,v)=\lim_{m\to\infty}\mathcal{E}_m(u,v),
\end{equation} 
and $dom\, \mathcal{E}\mod $ constants becomes a Hilbert space with this inner product.

The standard Lebesgue measure on the circle is pushed forward by $\gamma$ to the standard self-similar probability measure $\mu$ on the pentagasket. We may define the Laplacian via the weak formulation
\begin{equation}
\mathcal{E}(u,v)=-\int (\Delta u)vd\mu,
\end{equation}
for all $v\in dom\,\mathcal{E}$. Note that this is the Neumann Laplacian if we choose a boundary for the pentagasket. Typically one takes the boundary to be either the five points of the initial star, or just three of the five. It is important to observe that for PG, in contrast to the SG, there is nothing special about the local geometry in a neighborhood of a boundary point: there are infinitely many points with the same local geometry. Because Neumann boundary conditions are "natural," the Laplacian behaves the same way at all these locally isometric points regardless of which ones we designate as the boundary.

Of course we want a simpler method to give $\Delta u$ as a limit of discrete graph Laplacians $\Delta_m$ on $\Gamma_m$, and for this we need to approximate the measure $\mu$ by a discrete measure on the vertices of $\Gamma_m$ to get a formula of the form (1.1). For each point of the form $\frac{k}{5^{m+1}}$ we assign the weight $\frac{1}{5^{m+1}}$ if $k\equiv 0\mod 5$, and $\frac{2}{5^{m+1}}$ if $k\equiv 1$ or $4\mod 5$, and $\mu(x)$ is the sum of the weights of all the points in the equivalence class $x$. Of course the equivalence class just consists of the singleton $\frac{k}{5^m+1}$ if $k\equiv 0\mod 5$, so the $\mu(x)=\frac{1}{5^{m+1}}$, while if $k\equiv 1$ or $4\mod 5$, then the equivalence class consists of two points of the same type, so $\mu(x)=\frac{4}{5^{m+1}}$.

\begin{table}[htb!]
\begin{adjustwidth}{-1.2in}{-1in}
\begin{center}
$\begin{array}{|c|c|c|c|c|c|c|c|c|c|c|c|c} 
\hline
\multicolumn{3}{|c|}{\text {Level 1}}
& \multicolumn{3}{|c|}{\text{Level 2}}
& \multicolumn{3}{|c|}{\text{Level 3}}
& \multicolumn{3}{|c|}{\text{Level 4}}\\
\hline
\#  & \text{Mult} & \text{Eigenvalue} &  \# & \text{Mult} & \text{Eigenvalue}& \# & \text{Mult} & \text{Eigenvalue}& \# & \text{Mult} & \text{Eigenvalue}\\
\hline 1 & 1 & 0 & 1 & 1 & 0 & 1&1& 0 & 1 & 1 & 0\\ \hline 
\hline 2 & 2 & 28.6410 & 2 & 2 & 12.5186 &2&2& 12.6700&2& 2 &12.6832\\ \hline 
\hline 4&2&28.9251&4&2&30.6109 &4&	2&31.3706 &4&2 &31.4492 \\ \hline
\hline  6&2&119.5409&6 &5& 143.2049&6& 5& 135.7523&6& 5& 137.4025  \\  \hline 
\hline 8&2&132.5555&11 &1& 168.8936&11& 1& 164.5714&11& 1& 166.9378\\ \hline
\hline 10&1&135.5536&12&2&182.4264&12& 2& 182.3916&12& 2 &185.2678\\ \hline
\hline &&& 14&2&215.2990&14&2& 239.2249&14& 2& 	244.1480\\ \hline
\hline  &&& 16&5&415.7326&16&5& 331.9515&16& 5& 340.1929\\ \hline
\hline &&& 21&2&430.6319&21& 2& 435.5986&21& 2& 453.4902\\ \hline
\hline  &&& 23&2&454.5580&23& 2& 562.4423&23& 	2 	&596.8892 \\ \hline
\hline &&&	25&1&463.5525&25& 1 &629.634&25& 1 &677.4916\\ \hline
\hline  & && 26&5&597.7066 &26&20& 1552.9561&26&20&1472.1417\\ \hline
\end{array}$
\vspace{5mm}
\caption{ Eigenvalues of the Pentagasket}
\end{center}
\end{adjustwidth}
\end{table}

In defining the discrete Laplacian via (1.1), we renormalize the conductances by $r^{-m}$ to make $\Delta=\lim_{m\to\infty}\Delta_m$ without further renormalization factors. In Table 3.1 we show the eigenvalues with multiplicities for $m=$1, 2, 3, and $4$. It is clear that the eigenvalues are converging to a limit as $m\to\infty$, and the multiplicities and values up to a constant multiple agree with those computed in \cite{Spec}.

In Figures 3.1 and 3.2 we show graphs of eigenfunctions for $m=$2, and 3 as functions on the circle (here [0,1] with $0\equiv 1$) that respect identifications. We emphasize that this gives a new way to visualize eigenfunctions. The requirement that the functions respect identifications is a very stringent requirement. There are essentially no "nice" functions (for example, differentiable functions) with this property except for constants, and it is possible to visually recognize the patterns characteristic of these graphs.

\begin{table}[htb!]
\centering
\begin{adjustwidth}{-1.2in}{-1.0in}

\begin{tabular}{  c c c c }
$ \text{Eigenfunction \# 2 \& 3} $ &   $\text{Eigenfunction \# 4 \& 5}$ \\
\includegraphics[height=1.25in,width=1.75in]{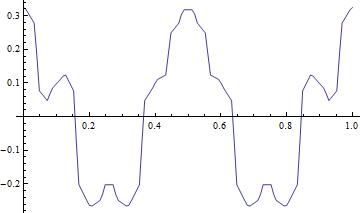}  \includegraphics[height=1.25in,width=1.75in]{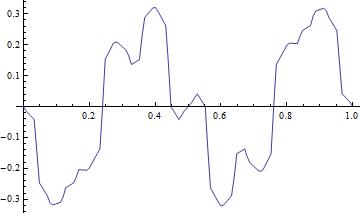}&
 \includegraphics[height=1.25in,width=1.75in]{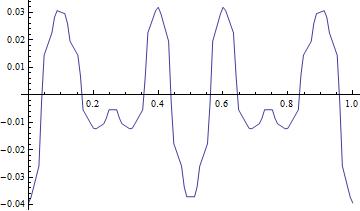}  \includegraphics[height=1.25in,width=1.75in]{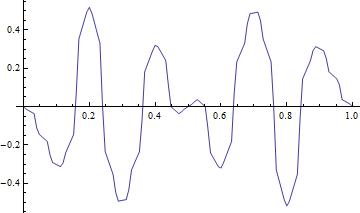}\\ 
\\
 $\text{Eigenfunction \# 11} $ \\
 \includegraphics[height=1.25in,width=1.75in]{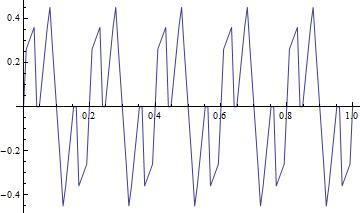}& \\
\end{tabular}
\end{adjustwidth}
\caption*{FIGURE 3.1. Eigenfunctions of the Pentagasket at Level 2}
\end{table}

\begin{table}[h!]
\centering
\begin{adjustwidth}{-1.2in}{-1.0in}

\begin{tabular}{  c c c c }
$ \text{Eigenfunction \# 2 \& 3} $ &   $\text{Eigenfunction \# 4 \& 5}$ \\
\includegraphics[height=1.25in,width=1.75in]{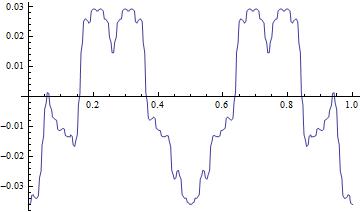}  \includegraphics[height=1.25in,width=1.75in]{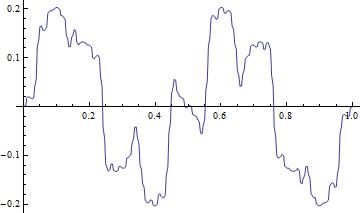}&
 \includegraphics[height=1.25in,width=1.75in]{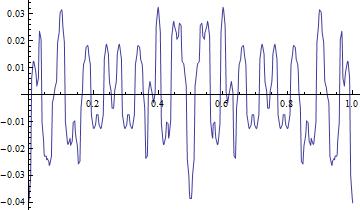}  \includegraphics[height=1.25in,width=1.75in]{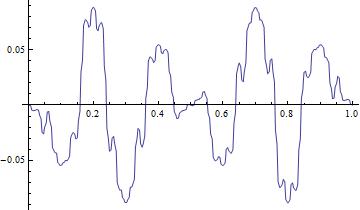}\\ 
$\text{Eigenfunction \#6}$ & $\text{Eigenfunction \#11}$\\
\includegraphics[height=1.25in,width=1.75in]{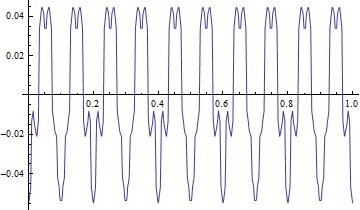}&\includegraphics[height=1.25in,width=1.75in]{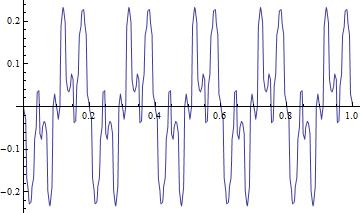}
\\
$\text{Eigenfunction \# 16}  $ & $\text{Eigenfunction \# 51} $ \\
\includegraphics[height=1.25in,width=1.75in]{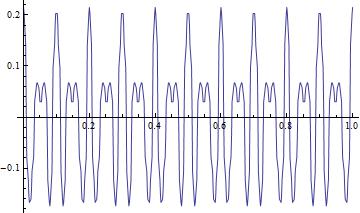}&
\includegraphics[height=1.25in,width=2.25in]{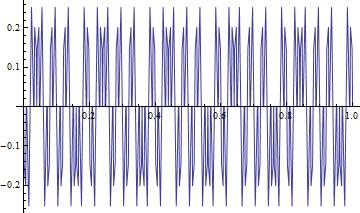} \\
\\
\end{tabular}

\caption*{FIGURE 3.2. Eigenfunctions of the Pentagasket at Level 3}
\end{adjustwidth}
\end{table}

The connection with the dihedral-5 group of symmetries is very straightforward in this realization. Rotations through the angle $\frac{2\pi j}{5}$ in PG correspond to translations $t\to t+\frac{j}{5}$, and reflections correspond to $t\to \frac{j}{5}-t$. For eigenspaces of multiplicity two, such as $\#$2 $\&$3 and $\#$4 $\&$5, we can find a basis of eigenfunctions satisfying the symmetry condition $u(1-t)=u(t)$ and the skew-symmetry condition $u(1-t)=-u(t)$, and we have automatically enforced this dichotomy in our choice of graphs. Eigenfunctions corresponding to multiplicity one, such as $\#$11, will have periodicity $u\left (t+\frac{1}{5}\right)=u(t)$, and will be skew-symmetric with respect to all reflections ( except in the trivial case of constants), as was shown in \cite{Spec}. Within eigenspaces of multiplicity five it is possible to find eigenfunctions that are symmetric with respect to both translations and reflections, such as $\#$6 (out of $\#$ 6-10). All these symmetries are immediately visible from the graphs.

We can also see miniaturization of eigenfunctions. Consider an eigenspace of multiplicity two with eigenvalue $\lambda$. Then $5r^{-1}\lambda$ will be an eigenspace of multiplicity five, and if $u$ is the reflection symmetric $\lambda$-eigenfunction then $u(5t)$ is the reflection symmetric periodic $5r^{-1}\lambda$ eigenfunction. This fact is proven in \cite{Spec}, but is visually obvious from the graphs of $\#$ 2 and $\#$ 4 on level 2 and $\# 6$ and $\# 16$ on level 3. Similarly, if $u$ is a $\lambda$-eigenfunction for an eigenspace of multiplicity one with $\lambda\neq 0$, then it is shown in \cite{Spec} that $u$ is reflection skew-symmetric and $u(5t)$ is a reflection skew-symmetric periodic $5r^{-1}\lambda$-eigenfunction, and this eigenspace has multiplicity five.This is seen in $\#11$ on level 2 and $\#51$ on level 3. Here we find the periodic eigenfunctions in the multiplicity five eigenspaces via periodization. 
In Figure 3.3 we show the log-log graphs of the eigenvalue counting function, $\rho(x)=\#\{\lambda_j \le x \}$ and the Weyl Ratio, $WR(x)=\frac{\rho(x)}{x^\beta}$ for $\beta=\frac{\text{log}5}{\text{log}5-\text{log}5}\approx 0.675$ on levels 2,3,and 4. We can begin to see evidence of the asymptotic multiplicative periodicity of the Weyl ratio, $WR(5x)\approx WR(x)$, in the level 4 graph
\vspace{.1in}

\begin{table}[h!]
\centering
\begin{adjustwidth}{-1.0in}{-1.0in}
\vspace{3mm}
\begin{tabular}{  c c c c }
$\text{Eigenvalue Counting Function}$\\
\\
$\text{ Level 2}$ & $\text{ Level 3}$ & $\text{Level 4}$ \\
\\
 \includegraphics[height=1.25in,width=1.75in]{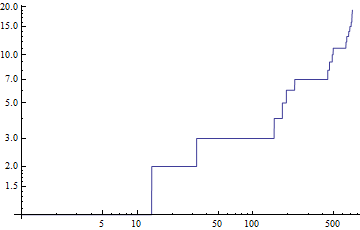}&  \includegraphics[height=1.25in,width=1.75in]{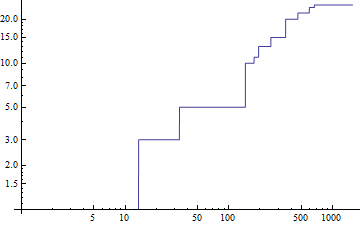} &  \includegraphics[height=1.25in,width=1.75in]{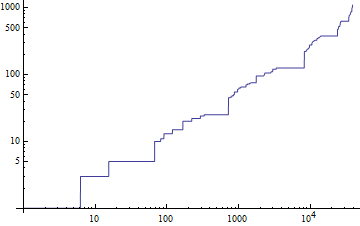}\\ 
\\
$\text{Weyl Ratios}$ \\
\\
$ \text{Level 2}$ & $ \text{Level 3}$ & $\text{Level 4} $\\
\\
 \includegraphics[height=1.25in,width=1.75in]{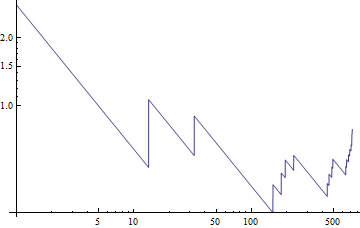} &  \includegraphics[height=1.25in,width=1.75in]{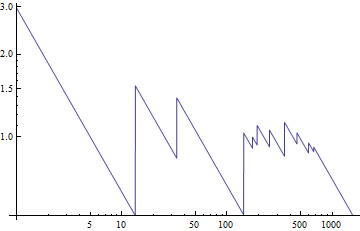} &  \includegraphics[height=1.25in,width=1.75in]{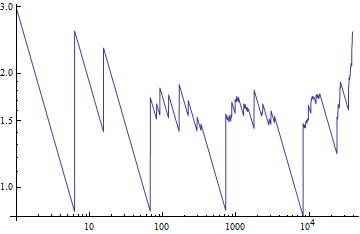}\\
\end{tabular}
\vspace{1mm}

\caption*{FIGURE 3.3. We show the log-log graphs of the eigenvalue counting function, $\rho(x)=\#\{\lambda_j \le x \}$ and the Weyl Ratio, $WR(x)=\frac{\rho(x)}{x^\beta}$ for $\beta=\frac{\text{log}5}{\text{log}5-\text{log}5}\approx 0.675$ on levels 2,3,and 4. We can begin to see evidence of the asymptotic multiplicative periodicity of the Weyl ratio, $WR(5x)\approx WR(x)$, in the level 4 graph.}
\end{adjustwidth}
\end{table}

\section{The Octagasket}

It is believed that the Octagasket (OG) has a symmetric self-similar energy $\mathcal{E}$ with $dom\mathcal{E}$ containing only continuous functions (equivalently, points have positive capacity), and an associated Laplacian defined by $(3.8)$ with the standard self-similar measure $\mu$. There is no proof for these conjectures at present. Experimental evidence for the existence of the Laplacian was provided in \cite{BHS}, and this paper will provide independent evidence. The Peano curve $\gamma$ and its approximations $\gamma_m$ give us a sequence of graph approximations $\Gamma _m$, and we can define a graph energy $E_m$ by giving single edges conductance 1 and double edges conductance $2$. We expect that there is an energy renormalization factor $r<1$ such that 
\begin{equation}
\mathcal{E}(u,v)=\lim_{m\to\infty} r^{-m} E_m (u,v)\text{ for }u,v\in dom\mathcal{E},
\end{equation}
and we will give a rough estimate for $r$ based on our experimental data. We note that OG is not PCF, and there is no clean formula relating $E_{m+1}(\bar u)$ and $E_{m}( u)$ when $\bar u$ is the $E_{m+1}$ minimizing extension of $u$ analogous to (3.3). It is natural to approximate $\mu$ by the discrete measure on vertices of $\Gamma_m$ that assigns weights $\frac{1}{2\cdot 8^{m+1}}$ to each point $\frac{k}{2\cdot 8^{m+1}}$, so outer vertices get weight $\frac{1}{8^{m+1}}$ and inner vertices get weight $\frac{2}{8^{m+1}}$. We define an unnormalized discrete Laplacian $\Delta_m$ on $\Gamma_m$ by 
\begin{equation}
-\Delta_m u(x)=4\left(u(x)-\text{Ave}(u(y))\right)
\end{equation}
where $\text{Ave}(u(y))$ denotes the average value of $u$ on the 2 or 4 neighboring points. (The constant 4 is just for convenience.) Note that this is a different convention than the one used in the case of PG. Now we expect
\begin{equation}
-\Delta u=lim_{m\to\infty} \left (\frac{8}{r} \right )^m\Delta_m u.
\end{equation}
The ratios of the corresponding eigenvalues for $\Delta_m$ and $\Delta_{m+1}$ give an estimate for the renormalization factor $\frac{8}{r}\approx 14.9$, which means
 $r\approx 0.537$. In Table 4.1 we give the eigenvalues for $m=1,2,3$ and the ratios.
In Table 4.2 we give the same data for the renormalized eigenvalues, where we have multiplied by $(\frac{8}{r})^m$.
It is clear that the multiplicities remain the same as $m$ increases (at least in the lower portion of the spectrum). The values are in reasonable agreement with those found in \cite{BHS}.

\begin{table}
\begin{adjustwidth}{-.5in}{-.5in}
\begin{center}
$\begin{array}{|c|c|c|c|c|c|c|c|c|||c|c|c|} 
\hline
\multicolumn{3}{|c|}{\text {Level 1}}
& \multicolumn{3}{|c|}{\text{Level 2}}
& \multicolumn{3}{|c|}{\text{Level 3}}
&\multicolumn{2}{|c|}{\text{Ratio}} \\
\hline
\#  & \text{Mult} & \text{Eigenvalue} &  \# & \text{Mult} & \text{Eigenvalue}& \# & \text{Mult} & \text{Eigenvalue}& \frac{\lambda_1}{\lambda_2}&\frac{\lambda_2}{\lambda_3}\\
\hline 1 & 1 & 0 & 1 & 1 & 0 & 1&1&0&&\\ \hline 
\hline 2 & 2 & 0.111 & 2 & 2 & 0.0074& 2 & 2 & 0.0005 &14.802 & 14.938\\ \hline 
\hline 4 & 2 & 0.396 & 4 & 2 & 0.0282 & 4 & 2 & 0.0018 & 14.027 & 14.908\\ \hline
\hline 6 & 2 & 0.770& 6 & 2 &  0.0570  & 6 & 2 & 0.0038 & 13.495 & 14.897\\  \hline 
\hline 8 & 3 & 1.171 & 8 & 1 & 0.0784 & 8 & 1 &0.0052 &&14.960\\ \hline
\hline 11 &2 & 1.276 & 9 & 2 & 0.1108 & 9 & 2 & 0.0074 &&14.794\\ \hline
\hline 13 & 2 & 1.500 & 11 & 2 & 0.1157&11&2&0.0077&&14.852\\ \hline
\hline 15 & 2 & 1.506& 13 & 2 & 0.1251&13&2&0.0083&&14.941\\ \hline
\hline 17 & 2 & 3.109 & 15 & 2 & 0.1263&15&2&0.0084&&14.971\\ \hline
\hline 19 & 2 & 3.299 & 17 & 2 & 0.2291&17&2&0.0154&&14.803\\ \hline
\hline 21 & 2 &3.465 & 19 & 1& 0.2362&19&1&0.0157&&15.034\\ \hline
\hline 23 & 4  & 4.000 & 20 & 2 & 0.2412&20&2&0.0165&& 14.590\\ \hline
\hline 27& 2 &4.534 & 22 & 2 &  0.2771  &22&2&0.0189&&14.605 \\ \hline
\hline 29 & 2 & 4.700 & 24 & 1 & 0.3021&24&1&0.0205&&14.691 \\ \hline
\hline 31 & 2 &4.890 & 25 & 2 & 0.3961&25&2&0.0282&&14.027 \\ \hline
\hline 33 & 2 & 6.493 & 27 & 2 & 0.4237 &27&2&0.0300&&14.120 \\ \hline
\hline 35 & 2 & 6.499  &29&2&0.4261&29&2&0.0301&&14.136  \\ \hline
\hline 37 & 2 & 6.723 & 31& 2 &0.4561& 31 &2&0.0321&&14.204 \\ \hline
\hline 39 & 3 & 6.828 & 33&2& 0.5912&33&2&0.0425&&13.909\\ \hline
\hline 42 & 2 & 7.229 & 35 & 2 & 0.5984&35&2&0.0428&&13.970\\ \hline
\hline 44 & 2 &  7.603 &37& 1&  0.6249&37&1&0.0445&&14.035\\ \hline
\hline 46 & 2 &  7.889 & 38 &2& 0.6650&38&2&0.0479&&13.871\\ \hline
\hline 48           & 1 & 8.000  &40&1 & 0.7536&40&1&0.0542&&13.903\\ \hline
\hline &&&41& 2 &0.7700 &41& 2 &0.0570&&13.495\\ \hline
\hline &&&43&2&0.8100&43&2&0.0598&&13.526\\ \hline
\hline &&&45&2&0.8525&45&2&0.0631&&13.504\\ \hline
\hline &&&47&2&0.8866&47&2&0.0656&&13.507\\ \hline
\hline &&&49&2&0.9328&49&2&0.0696&&13.401\\ \hline
\hline &&&51&2&0.9772&51&2&0.0729&&13.389\\ \hline
\hline &&&53&2&0.9891&53&2&0.0735&&13.446\\ \hline
\hline &&&55&1&1.0151&55&1&0.0750&&13.532\\ \hline
\hline &&&56&1&1.0314&56&1&0.0784&&13.146\\ \hline
\hline &&&57&3&1.1715&57&1&0.0843&&\\ \hline
\hline &&&60&2&1.1810&58&1&0.0901&&\\ \hline
\hline &&&62&2&1.1933&60&2&0.0973&&\\ \hline
\hline &&&64&2&1.2025&62&2&0.1024&&\\ \hline
\hline &&&66&2&1.2201&64&1&0.1042&&\\  \hline
\end{array}$
\vspace{5mm}
\caption{Eigenvalues of OG}
\end{center}
\end{adjustwidth}

\end{table}

\begin{table}
\centering
$\begin{array}{|c|c|c|c|c|c|c|c|c|c|c|}
\hline
\multicolumn{3}{|c|}{\text {Level 1}}
& \multicolumn{3}{|c|}{\text{Level 2}}
& \multicolumn{3}{|c|}{\text{Level 3}} \\
\hline
\#  & \text{Mult} & \text{Eigenvalue} &  \# & \text{Mult} & \text{Eigenvalue}& \# & \text{Mult} & \text{Eigenvalue} \\
\hline 1 & 1 & 0 & 1 & 1 & 0 & 1&1&0\\ 
\hline 2 & 2 &1.652 & 2 & 2& 1.662 & 2& 2 &  1.659 \\ \hline 
\hline 4 & 2 & 5.902 &4& 2 & 6.269 & 4 & 2 & 6.265\\ \hline
\hline 6 & 2 &  11.473 & 6 & 2 & 12.667 & 6  & 2 &  12.669 \\  \hline 
\hline 8 & 3 & 17.456&8 & 1 & 17.418 & 8& 1 & 17.348 \\ \hline
\hline 11 & 2 & 19.018 &9 & 2 & 24.614 & 9 & 2 & 24.789 \\ \hline
\hline 13 & 2 & 22.355 & 11& 2 &25.693 & 11& 2&25.775  \\ \hline
\hline 15 & 2 & 22.439 & 13 & 2 & 27.773 & 13 & 2 & 27.697 \\ \hline
\hline 17 & 2 &  46.337 & 15& 2 & 28.039& 15 & 2 &  27.905\\ \hline
\hline 19 & 2& 49.160 & 17 & 2 &  50.878 & 17 & 2 & 51.209 \\ \hline
\hline 21 & 2 &  51.642 & 19 & 1 &  52.445 & 19& 1 & 51.974 \\ \hline
\hline 23       & 4 & 59.600  & 20 & 2 &  53.551 & 20 & 2&  54.688 \\ \hline
\hline 27    & 2 &  67.557 & 22 & 2 & 61.527 & 22& 2 & 62.769\\ \hline
\hline 29 & 2 &  70.039 & 24& 1 &  67.071 & 24& 1 & 68.021\\ \hline
\hline 31  & 2 & 72.862 & 25 & 2 &   87.942&25 & 2 & 93.410\\ \hline
\hline 33 & 2 & 96.760 & 27 & 2 &  94.083  &27 & 2 & 99.274  \\ \hline
\hline 35 & 2 & 96.844 &29 & 2 &  94.616 &29 & 2 & 99.727  \\ \hline
\hline 37 & 2 & 100.181  & 31& 2 &101.269 & 31 & 2 & 106.231 \\ \hline
\hline 39 & 3 & 101.743 &33 & 1 &  131.263 &33& 2 & 140.612\\ \hline
\hline 42 & 2 & 107.726 & 35 & 2 & 132.857  & 35 & 2 & 141.695\\ \hline
\hline 44 & 2 & 113.297 & 37& 1 & 138.734 & 37 & 1 & 147.282\\ \hline
\hline 46 & 2 & 117.548 & 38& 2 &  147.636& 38 & 2 &158.585\\ \hline
\hline 48 & 1 & 119.200 & 40  &  1 & 167.317 & 40 & 1 & 179.305\\ \hline
\hline &&&41&2 & 170.961 & 41 & 2 & 188.748 \\ \hline
\hline &&&43&2&179.830&43&2&198.093\\ \hline
\hline &&&45&2&189.270&45&2&208.827\\ \hline
\hline &&&47&2&196.836&47&2&217.132\\ \hline
\hline &&&49&2&207.097&49&2&230.250\\ \hline
\hline &&&51&2&216.965&51&2&241.441\\ \hline
\hline &&&53&2&219.590&53&2&243.324\\ \hline
\hline &&&55&1&225.377&55&1&248.161\\ \hline
\hline &&&56&1&228.994&56&1&259.535\\ \hline
\hline &&&57&3&260.100&57&1&278.937\\ \hline
\hline &&&60&2&262.198&58&2&298.082\\ \hline
\hline &&&62&2&264.926&60&2&322.073\\ \hline
\hline &&&64&2&266.978&63&2&338.988\\ \hline
\hline &&&66&2&270.885&64&1&344.883\\  \hline
\end{array}$
\vspace{5mm}
\caption{Renormalized eigenvalues of OG}
\end{table}

There are certain patterns to the eigenvalues of $-\Delta_m$ that are quite striking. The first is that, since the graph $\Gamma_m$ is bipartite (the even and odd numbers of $k$ in $\frac{k}{2\cdot 8^m}$ alternate), eigenvalues come in pairs: if $-\Delta_m u=\lambda u$, then

\begin{equation}
-\Delta_m u^*=\left(8-\lambda \right )u^*,\text{ where}
\end{equation}

\begin{equation}
u^*\left (\frac{k}{2\cdot 8^m} \right)=(-1)^k u \left (\frac{k}{2\cdot 8^m} \right),
\end{equation}
It does not seem that this observation has any consequences for the spectrum of $-\Delta$, since only the lower portion of the spectrum of $-\Delta_m$ is relevant. A more significant observation is the miniaturization of eigenfunctions: each eigenvalue of $-\Delta_m$ is also an eigenvalue of $-\Delta_{m+1}$ with the same multiplicity (with three exceptions to be explained below), and the corresponding eigenfunction of $-\Delta_m$ is "miniaturized" to create the eigenfunctions of $-\Delta_{m+1}.$ The rule for miniaturization depends on the representation of the dihedral-8 symmetry group that the eigenspace corresponds to, as explained in \cite{BHS}. There are three 2-dimensional representations, labeled $2_1,\,2_2,\,2_3$ and four 1-dimensional representations labeled $1\pm \pm$ for symmetry or skew symmetry with respect to the reflections through the centers of edges of the octagon and the reflections through the vertices of the octagon. The 2-dimensional representations miniaturize to representations of the same type, while for the 1-dimensional representations, the miniaturization rule is 
\begin{equation}
\begin{cases}
1++\to1++\\
1+-\to 1+-\\
1-+\to1++\\
1--\to1+-
\end{cases}
\end{equation}

For an example of the first rule, a constant miniaturizes to a constant. An example of the third rule is illustrated by Figure 4.1.\\

\begin{figure}[h!]
\begin{minipage}[h!]{0.48\linewidth}
\begin{center}
\includegraphics[trim = 23mm 10mm 0mm 0mm, clip,height=2.5in]{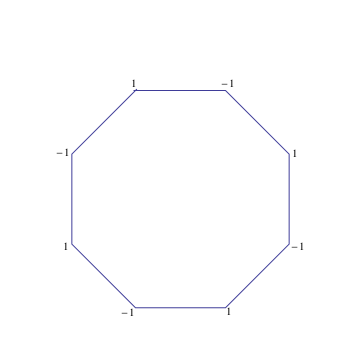}
\end{center}
m=0
\end{minipage}
\hspace{0.01cm}
\begin{minipage}[h!]{0.48\linewidth}
\begin{center}
\includegraphics[trim = 15mm 15mm 0mm 0mm, clip,height=2.5in]{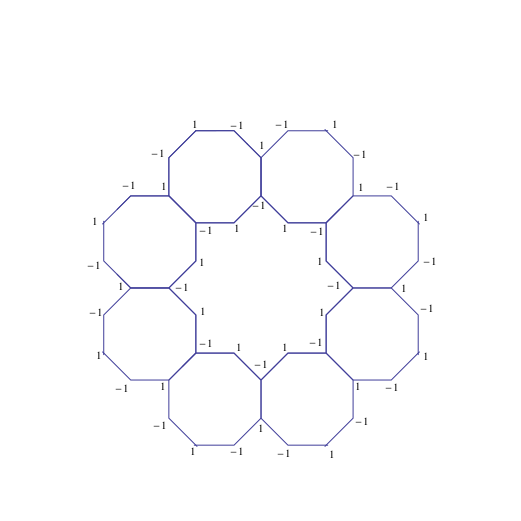}
\end{center}
m=1
\end{minipage}
\caption{Shows the miniaturization rule for the third symmetry type, $1-+\to 1+-$.}
\label{fig:figure4.1}
\end{figure}

%
%
%

(Since the $1+-$ and $1-+$ representations do not occur for m=0, it is not practical to illustrate the other two cases.)

The exceptional miniaturizations correspond to the eigenvalues $4\pm2\sqrt{2}$ and $4$. In fact the multiplicity of $4\pm2\sqrt{2}$ is three for each $m\geq 1$. Figure 4.2 shows the $m=1$ case with symmetry types $1++$, $1+-$ and $1-+$.


\begin{figure}[h!]
\includegraphics[trim = 35mm 35mm 35mm 35mm, clip,height=2.5in]{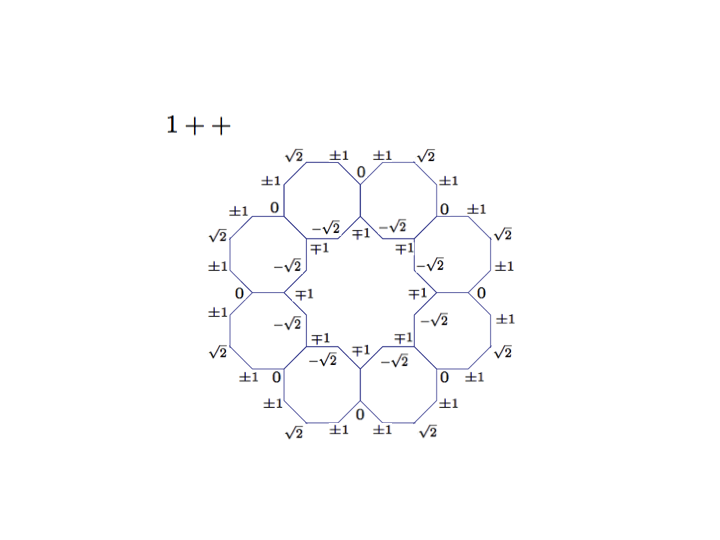}
\includegraphics[trim = 35mm 35mm 35mm 35mm, clip,height=2.5in]{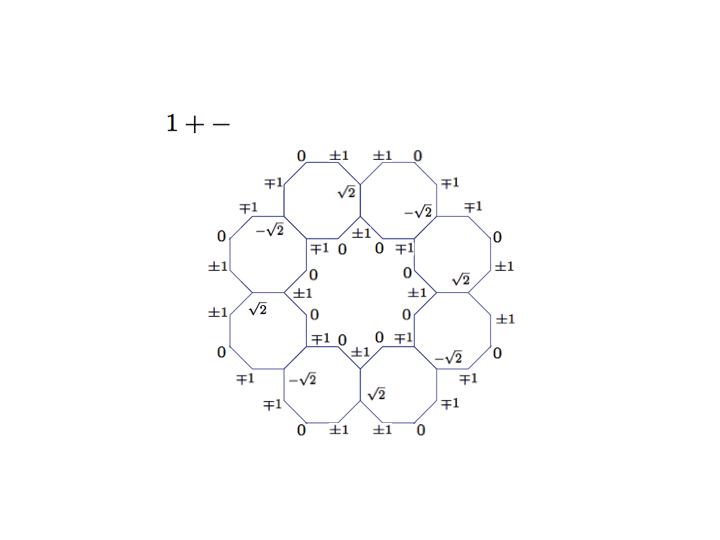}\\
\includegraphics[trim = 35mm 35mm 35mm 35mm, clip,height=2.5in]{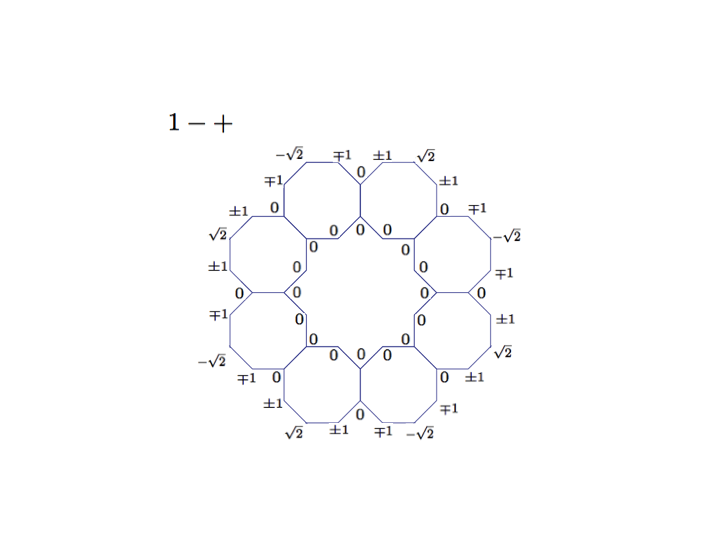}
\caption{Shows the $m=1$ case for symmetry types $1++$, $1+-$ and $1-+$ for eigenvalue $4\pm 2\sqrt{2}$.}
\label{fig:figure2}
\end{figure}

The $\lambda=4$ eigenspace has multiplicity that grows with $m$ (4, 20, 164 for $m=1,\,2,\,3$). 
Again, it does not appear that these exceptional eigenspaces contribute to the spectrum of $-\Delta$, which appears to only have multiplicities one and two.

We call an eigenvalue (or eigenspace) of $-\Delta_m$ \emph{primitive} if it is not an eigenvalue of $-\Delta_{m-1}$, otherwise we call it a \emph{derived} eigenvalue. In the limit we expect that the primitive eigenspaces of $-\Delta$ do not contain any miniaturized eigenfunctions. Just as in the case of PG, as proven in \cite{Spec}, there are restrictions of the types of 1-dimensional representations that can appear in primitive eigenspaces. 

\begin{table}
\centering
\begin{adjustwidth}{-1.0in}{-1.0in}
\begin{tabular}{  c c c c }
$\text{Eigenfunction \# 2 \& 3}$ & $\text{Eigenfunction \# 4 \& 5}$ \\
\\
\includegraphics[height=1.25in,width=1.75in]{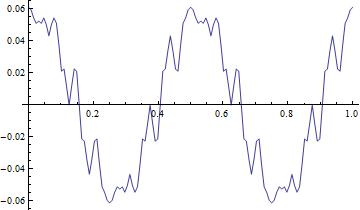}  \includegraphics[height=1.25in,width=1.75in]{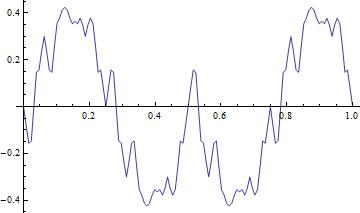}&
 \includegraphics[height=1.25in,width=1.75in]{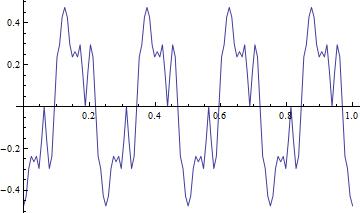}  \includegraphics[height=1.25in,width=1.75in]{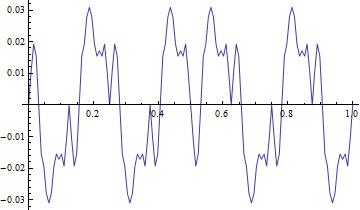}\\ 
\\
$ \text{Eigenfunction \# 6 \& 7}$ \\
\\
\includegraphics[height=1.25in,width=1.75in]{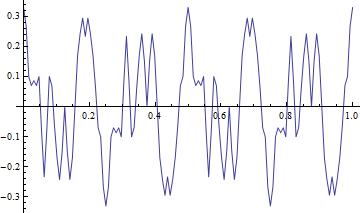} \includegraphics[height=1.25in,width=1.75in]{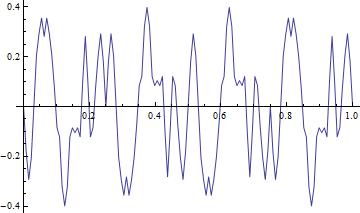}& 
\end{tabular}
\vspace{1mm}

\caption*{FIGURE 4.3: Eigenfunctions of OG at Level 1}
\end{adjustwidth}
\end{table}

\begin{theorem}
A primitive eigenfunction of $-\Delta_m$, other than a constant, cannot have $1++$ or $1+-$ symmetry.
\end{theorem}
\begin{proof}
Suppose $-\Delta_mu=\lambda u$ and $u$ is symmetric with respect to the reflections through the centers of the sides of the octagon. Then we could construct an eigenfunction $v$ of $-\Delta_{m-1}$ by blowing up $u$ on a one-cell, so $v(t)=u\left(\frac{1}{8}t\right)$ in the Peano curve parameterization, or $v(x)=u\left (F_j x \right)$, where $F_j$ is one of the contractions in the IFs that defines the octagasket. This would show that $u$ is not primitive. The only nontrivial fact to verify is that $v$ satisfies the eigenvalue equation along the "boundary" of $\Gamma_{m-1}$, since for "interior" vertices the equations $-\Delta_m u=\lambda u$ and $-\Delta_{m-1} v=\lambda v$ are identical. But this is exactly where the symmetry hypothesis comes in: in $\Gamma_m$ there will be four neighbors and in $\Gamma_{m-1}$ there will be two neighbors, but the average over the neighbors will be the same.
\end{proof}

The Peano curve we are using does not respect all the symmetries of the dihedral-8 symmetry group. However, the reflection in the line through the initial point $\gamma(0)$ is represented by the transformation $t\to t+\frac{1}{2}$, so we can sort eigenfunctions into symmetric and skew-symmetric ones with respect to this symmetry. Of course one-dimensional eigenspaces will automatically be one or the other. The two-dimensional eigenspaces will have a basis consisting of one symmetric and one skew-symmetric. Figures 4.3 and 4.4 
show the graphs of typical eigenfunctions at levels 2 and 3. We observe immediately that primitive $2_1$ and $2_3$ type eigenspaces contain a symmetric eigenfunction satisfying $u \bigr( t+\frac{1}{4}\bigl) =-u(t)$, for example $\# 2$ and $\#6$, while the primitive $2_2$ type eigenspace contains symmetric eigenfunctions satisfying $u\bigr(t+\frac{1}{8}\bigl)=-u(t)$, hence $u\bigr(t+\frac{1}{4}\bigl)=u(t)$, for example $\# 4$. A primitive $1-+$ eigenfunction will satisfy $u\bigr(t+\frac{1}{16}\bigl)=-u(t)$, hence $u\bigr(t+\frac{1}{8}\bigl)=u(t)$ while a primitive $1--$ eigenfunction will satisfy $\tilde u \bigr( t+\frac{1}{16}\bigl)=\tilde u(t)$, where
\begin{equation}
\tilde{u}(t)=\begin{cases} u(t)\text{ if }0\leq t\leq\frac{1}{2}\\
-u(t)\text{ if }\frac{1}{2}\leq t\leq 1,
\end{cases}
\end{equation}

for example $\smaller \# \hspace{0.8mm}   24$. 

Miniaturization will decrease the periods by a factor of $\frac{1}{8}$. For example, $\# \hspace{0.8mm}9$ has period $\frac{1}{16}$ (coming from the period $\frac{1}{2}$ of $\#\hspace{0.8mm}2$) and $\#\hspace{0.8mm} 25$ has period $\frac{1}{32}$ ( coming from the period $\frac{1}{4}$ of $\#\hspace{0.8mm} 4$). In Figure 4.5 we show explicitly the miniaturization of  $ \smaller\#\hspace{0.8mm} 2  \hspace{0.8mm}\&\hspace{0.8mm} 3$  at level 1 to $\#\hspace{0.8mm} 8 \hspace{0.8mm}\& \hspace{0.8mm}9$ at level 2. 

In Figure 4.6 we show the eigenvalue counting function and the Weyl ratio at levels 1,2, and 3, using the values $\beta=0.7213$ which was obtained experimentally. We do not observe any evidence of asymptotic multiplicative periodicity for the Weyl Ratio. 

\begin{table}
\centering
\begin{adjustwidth}{-1.2in}{-1.0in}

\begin{tabular}{  c c c c }
$ \text{Eigenfunction \# 2 \& 3} $ &   $\text{Eigenfunction \# 4 \& 5}$ \\
\includegraphics[height=1.25in,width=1.75in]{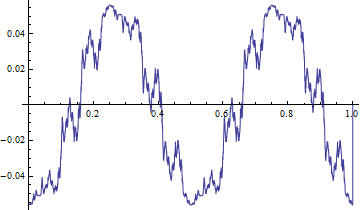}\includegraphics[height=1.25in,width=1.75in]{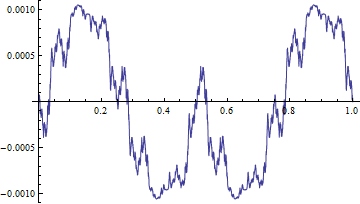}&
 \includegraphics[height=1.25in,width=1.75in]{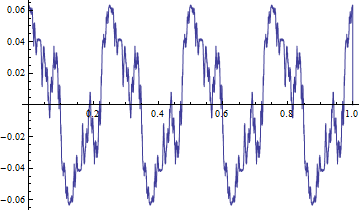}  \includegraphics[height=1.25in,width=1.75in]{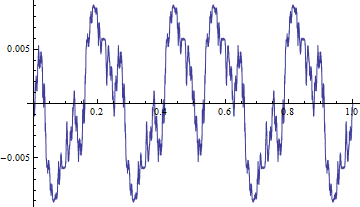}\\ 
\\
$\text{Eigenfunction \# 6 \& 7}  $ & $\text{Eigenfunction \# 8} $ \\
\includegraphics[height=1.25in,width=1.75in]{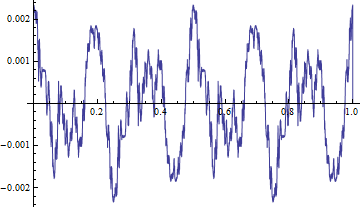} \includegraphics[height=1.25in,width=1.75in]{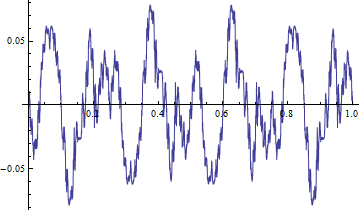}&
\includegraphics[height=1.25in,width=2.25in]{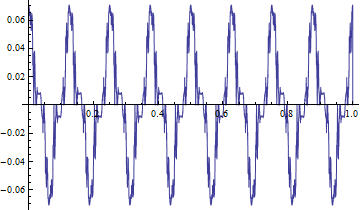} 
\\
$ \text{Eigenfunction \# 9 \& 10}  $ & $\text{Eigenfunction \# 24}$ \\
\includegraphics[height=1.25in,width=1.75in]{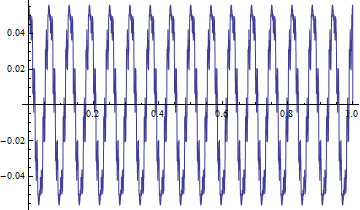} \includegraphics[height=1.25in,width=1.75in]{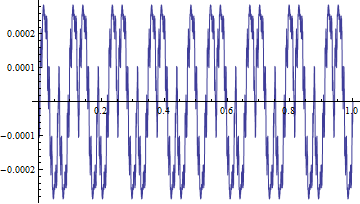}&
\includegraphics[height=1.25in,width=1.75in]{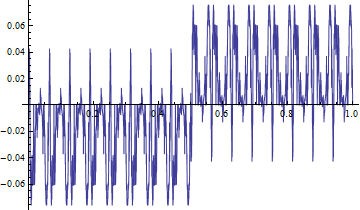}\\
\\
 $\text{Eigenfunction \# 25 \& 26}$ \\
\includegraphics[height=1.25in,width=1.75in]{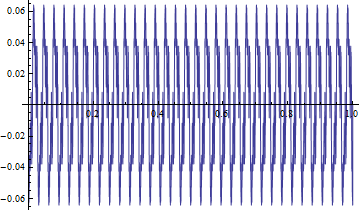} \includegraphics[height=1.25in,width=1.75in]{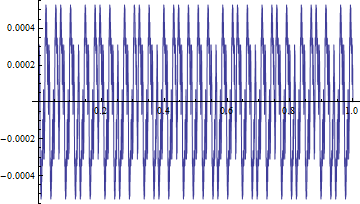}\\
\end{tabular}
\vspace{2mm}
\caption*{FIGURE 4.4: Eigenfunctions of the Octagasket at Level 3}
\end{adjustwidth}
\end{table}

\begin{table}
\centering
\begin{tabular}{ c c c  }
\includegraphics[height=1.25in,width=1.75in]{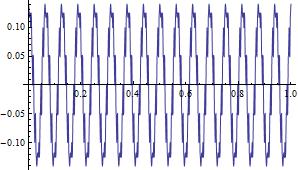} && \includegraphics[height=1.25in,width=1.75in]{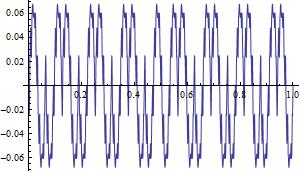}\\
\includegraphics[height=1.25in,width=1.75in]{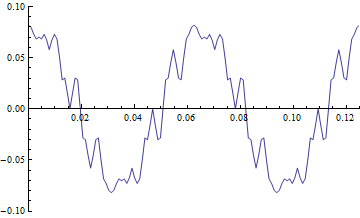} && \includegraphics[height=1.25in,width=1.75in]{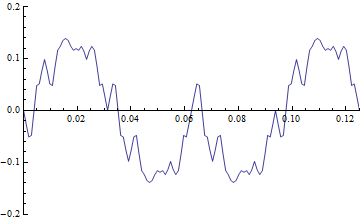}\\
\end{tabular}
\vspace{1mm}

\caption*{FIGURE 4.5. The top row shows $\#$ 9 $\&$10 at level 2, and the bottom row shows a zoom in by a factor of $\frac{1}{8}$, to be compared with $\#$ 2,3 on level 1 in Figure 4.3. }
\end{table}

It was observed in \cite{BHS} that the spectrum of OG appears to have spectral gaps (ratios $\frac{\lambda_{k+1}}{\lambda_k}$ considerably larger than 1 ). On the basis of our data we see gaps occurring for values of $k$ divisible by 16. In Table 4.3 we present this data.


\newpage 

\begin{table}
\centering
$\begin{array}{|c|c|c|} 
\hline k & 16k& \frac{\lambda_{16k+1}}{\lambda_{16k}}\\ \hline
\hline 1 & 16 & 1.8350\\ \hline
\hline 2 & 32 & 1.3236\\ \hline
\hline 7 & 112 & 1.554\\ \hline
\hline 15 & 240 & 1.168\\ \hline
\hline 54 & 864 & 1.768\\ \hline
\end{array}$
\vspace{5mm}
\caption{Spectral Gaps of the Octagasket at Level 3}
\end{table}

\begin{table}
\centering
\begin{adjustwidth}{-1.0in}{-1.0in}

\vspace{3mm}
\begin{tabular}{  c c c c c}
$\text{Eigenvalue Counting Functions}$&& \\
\\
$\text{ Level 1}$ & $\text{Level 2}$ & $\text{Level 3}$ \\
\\
  \includegraphics[height=1.25in,width=1.75in]{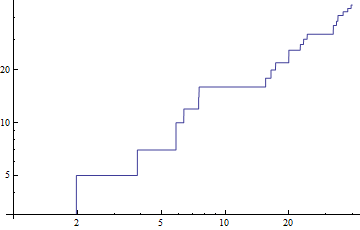}&  \includegraphics[height=1.25in,width=1.75in]{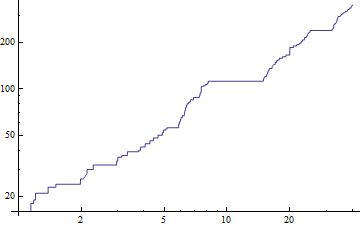}&  \includegraphics[height=1.25in,width=1.75in]{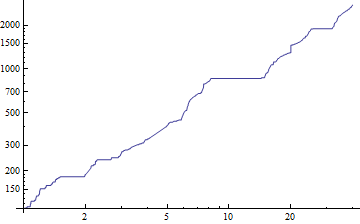}\\ 
\\
\\
$\text{Weyl Ratios}$ &&\\
\\
$ \text{Level 1} $ & $\text{Level 2}$ & $\text{Level 3}$ \\
\\
 \includegraphics[height=1.25in,width=1.75in]{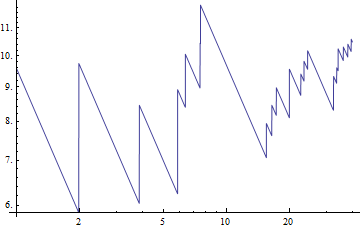} & \includegraphics[height=1.25in,width=1.75in]{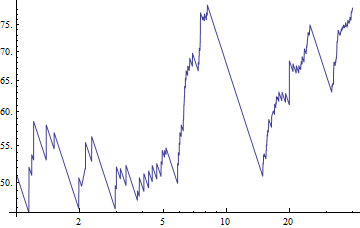} & \includegraphics[height=1.25in,width=1.75in]{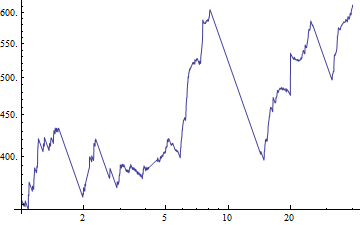}\\
\end{tabular}
\vspace{2mm}

\caption*{FIGURE 4.6. Eigenvalue counting function and Weyl Ratio of OG}
\end{adjustwidth}
\end{table}








\section{The Magic Carpet}

The approximations to $\gamma_m$ to the Peano Curve to MC give rise to a sequence of graphs $\Gamma_m$, and the graph Laplacians $\Delta_m$ are candidates for approximations to a Laplacian $\Delta$ on MC, with appropriate renormalization. We choose to normalize $\Delta_m$ by 
\begin{equation} 
-\Delta_m u(x)=12 (u(x) - \text{Ave}(u(y)) \end{equation} 
analogous to (4.2). Note, however, that some points $x$ will have neighbors that are identified with it, so more explicitly

\begin{equation}
 -\Delta_m (x) =
\begin{cases}
12u(x) - 3 \sum_{y\sim x} u(y) & \text{if $x$ has 2 identifications} \\
12 u(x) -\sum_{y\sim x} u(y) & \text{if $x$ has 6 identifications and 12 distinct neighbors}\\
8u(x) -\sum_{y\sim x} u(y) & \text{if $x$ has 6 identifications and 8 distinct neighbors}\\
\end{cases}
 \end{equation}

The third case occurs exactly when $x$ is a singular point introduced at level $m$, and the second case occurs when $x$ is a singular point introduced at level $m'<m$. We note that these approximate Laplacians agree exactly (except for a different renormalization constant) with the approximate Laplacians for zero-forms studied in \cite{kform}, and indeed the eigenvalues shown in Table 5.1 are equal to six times the eigenvalues computed in \cite{kform}. We would like to believe that a limit as in (4.3) exists for an appropriate choice of $r$. Experimentally it appears that $r\approx1.25$ and $\frac{8}{r} \approx 6.45$. This is very interesting because it means that the energy renormalization factor $r^m$ blows up. This would imply that the associated energy has points with zero capacity and functions of finite energy need not be continuous, as in Euclidean space of dimensions $ >2$. This would make it more challenging to define energy on MC as a limit of graph energies on $\Gamma_m$. The spectrum of $-\Delta_1$ is exactly $\{ 0, 9, \frac{29-\sqrt{73}}{2}, 15, 15, \frac{29 +\sqrt{73}}{2} \}$ and the associated non-constant eigenfunction are shown in Figures 5.1 -5.3.

\newpage

\begin{figure}[!ht]
\begin{center}
\includegraphics[height=2.5in,width=2.5in]{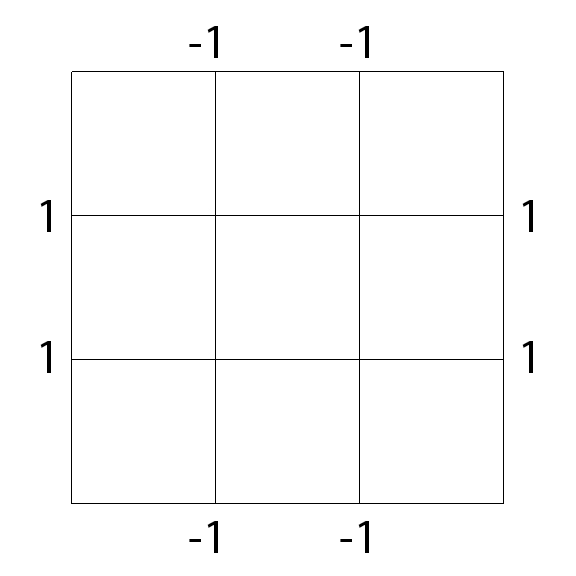}
\caption*{FIGURE 5.1. $\lambda=9$}
\end{center}
\end{figure}

\begin{figure}[!ht]
\begin{center}
\includegraphics[height=2.5in,width=2.5in]{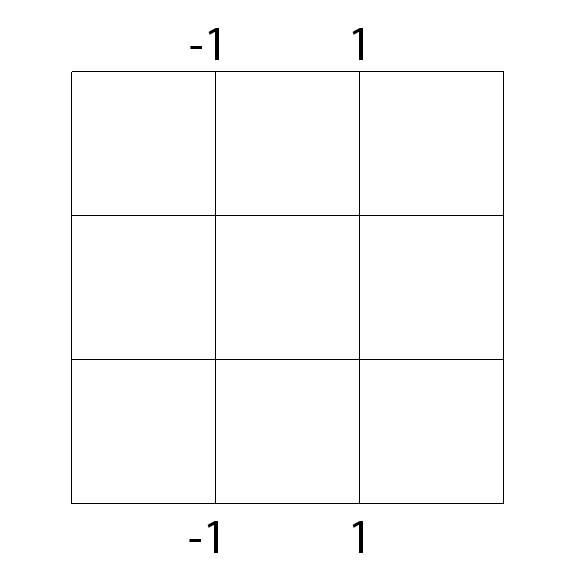}
\caption*{FIGURE 5.2. $\lambda=15$. Note that the eigenfunction may also be rotated.}
\end{center}
\end{figure}

\begin{figure}[!ht]
\begin{center}
\includegraphics[height=3in,width=3in]{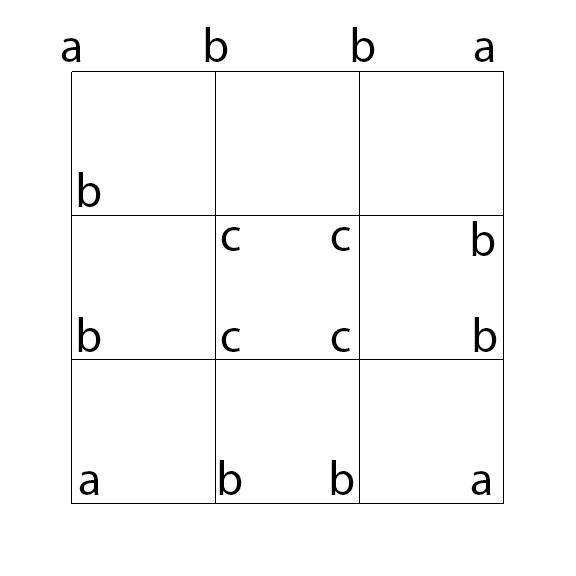}
\caption*{FIGURE 5.3. $\lambda=\frac{29 \pm \sqrt{73}}{2}$,$a=\frac{12}{\lambda-12}$,$b=-1$, $c=\frac{8}{\lambda-8}$}
\end{center}
\end{figure}

\begin{figure}[!ht]
\centering
\begin{minipage}[h!]{0.48\linewidth}
\begin{center}
\includegraphics[height=2.5in,width=2.5in]{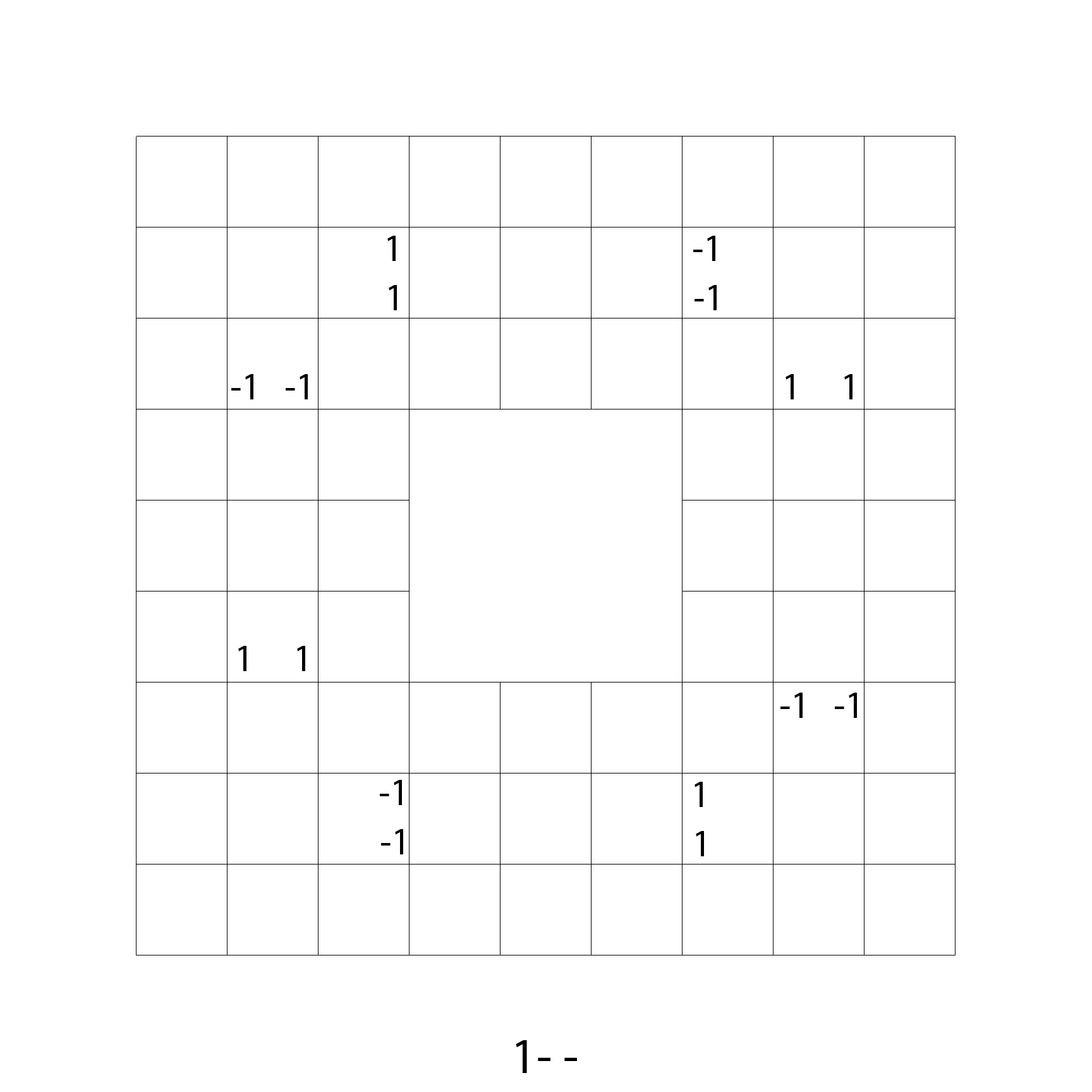}
\hspace{0.01cm}
\includegraphics[height=2.5in,width=2.5in]{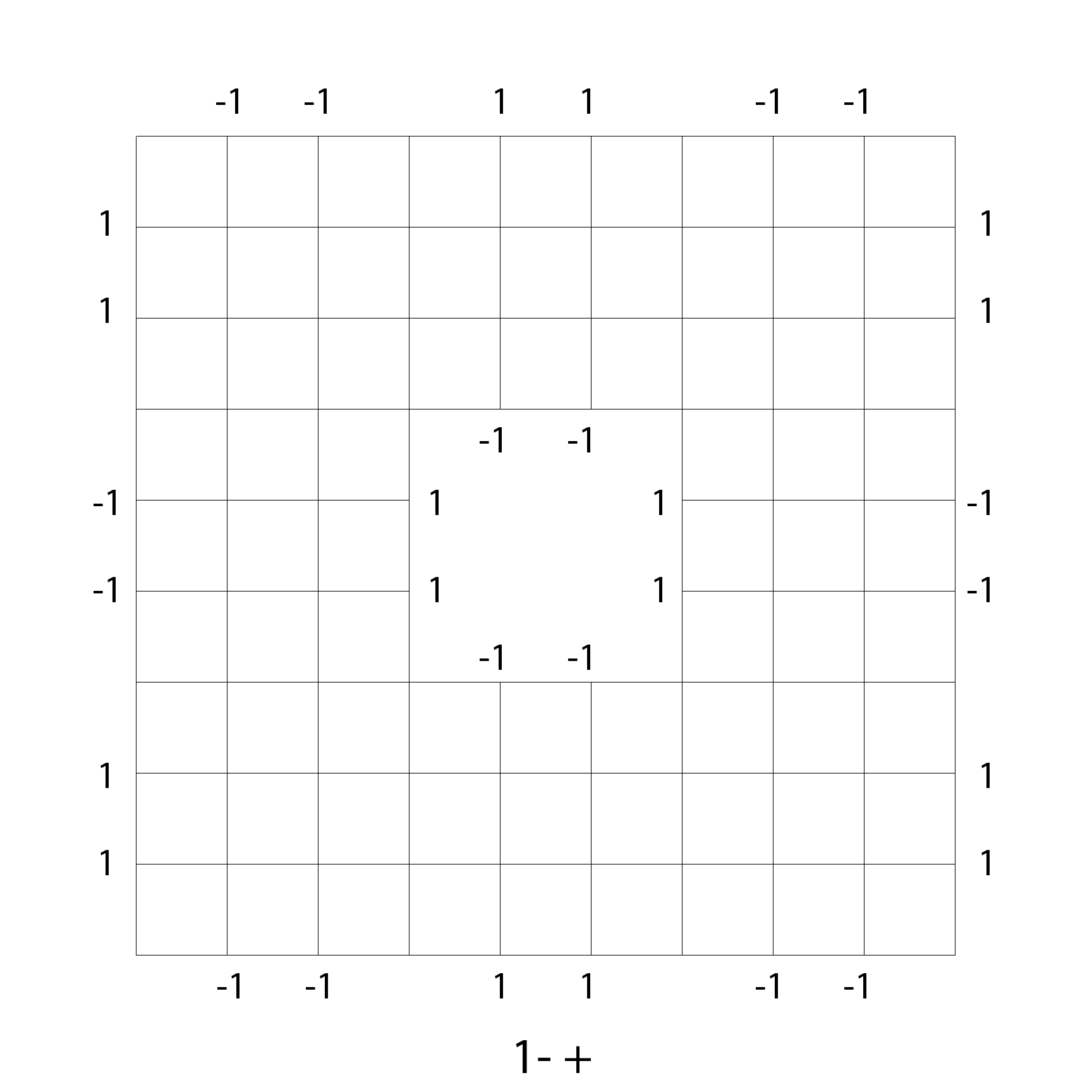}\\
\end{center}
\end{minipage}
\vspace{0.2cm}
\begin{minipage}[h!]{0.45\linewidth}
\begin{center}
\includegraphics[height=2.5in,width=2.5in]{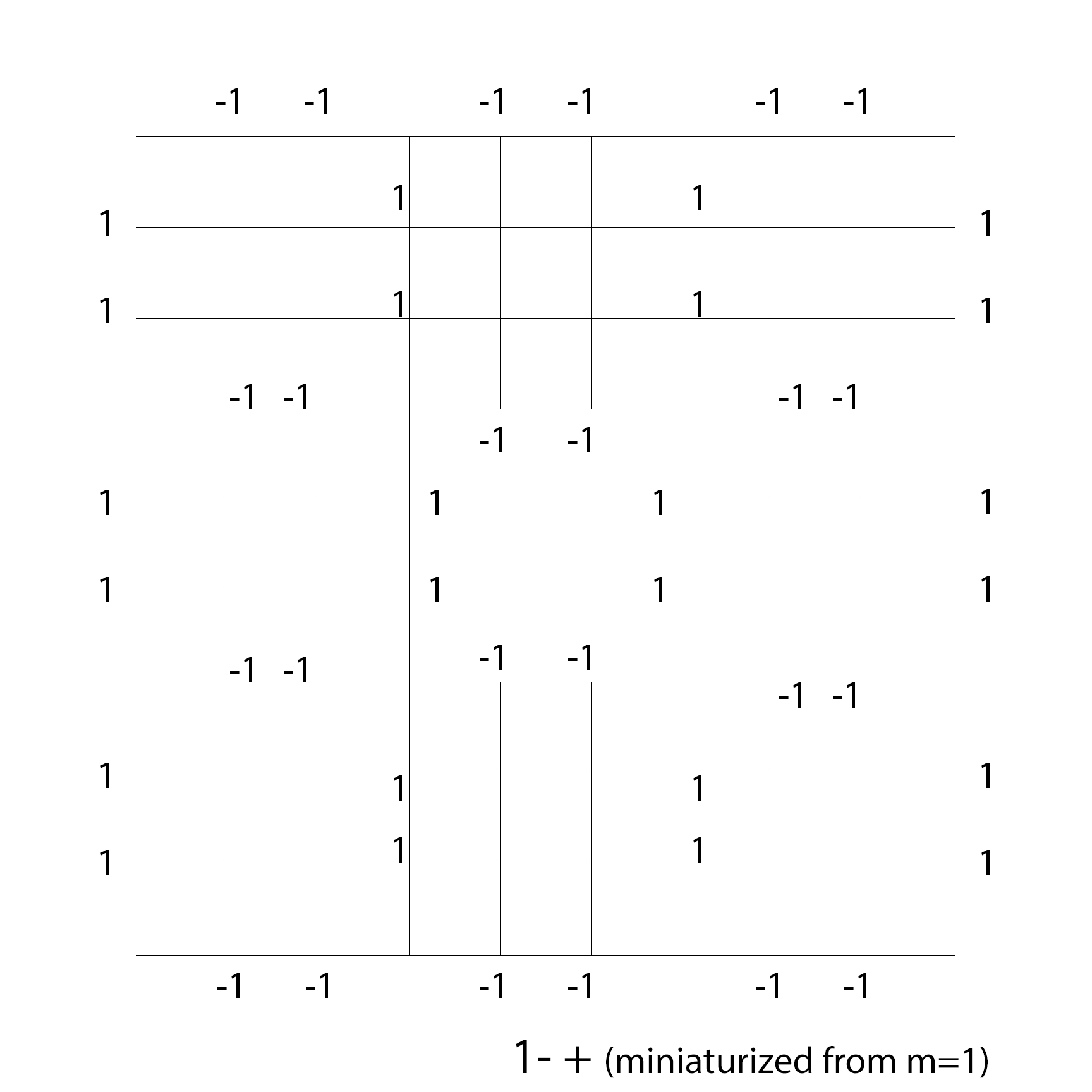}
\end{center}
\end{minipage}
\caption*{FIGRUE 5.4: Basis for $\lambda=9$ eigenspace for $m=2$ }
\label{fig:figure2}
\end{figure}

\begin{figure}[!ht]
\begin{center}
\includegraphics[height=3in,width=3in]{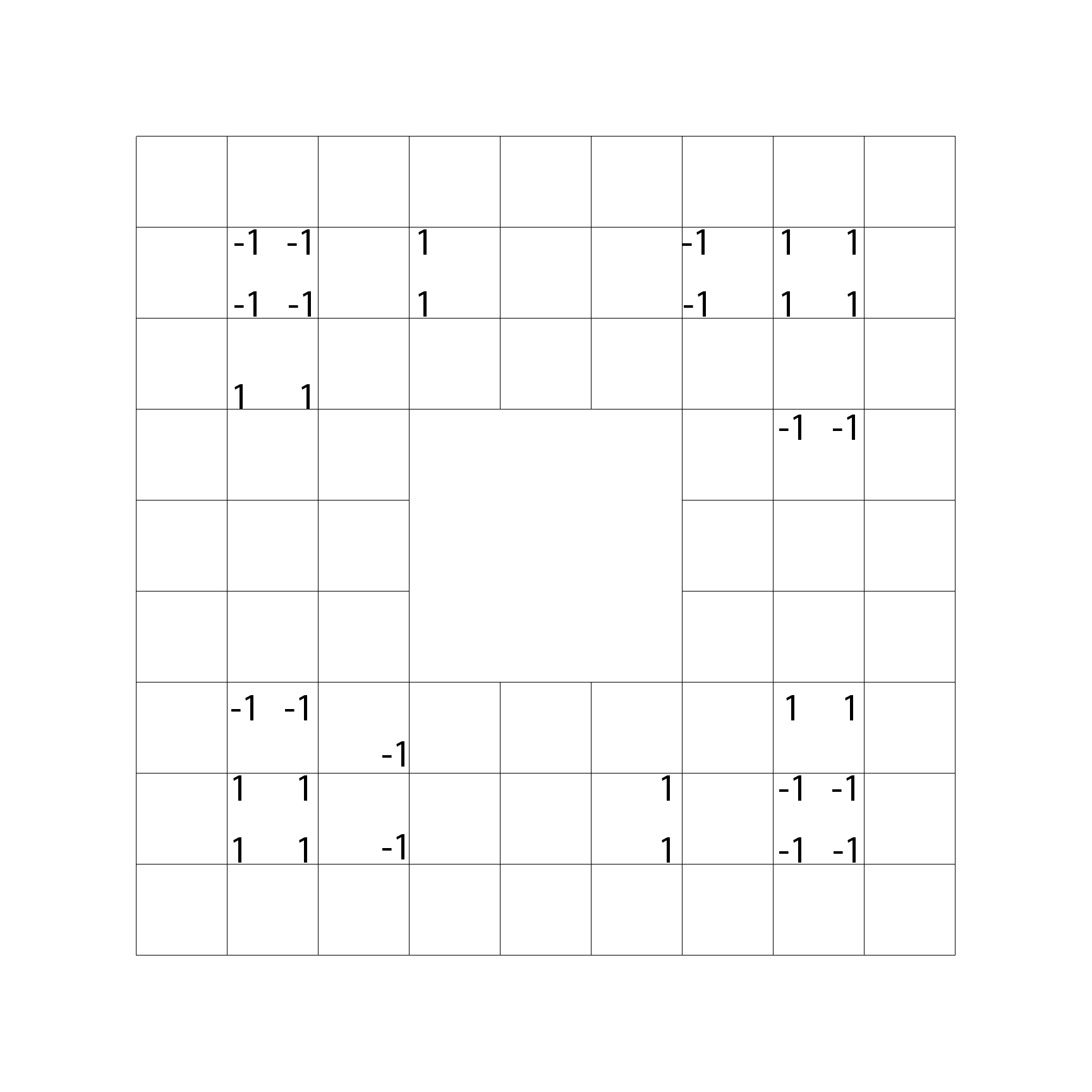}
\caption*{FIGURE 5.5:$\lambda=12$}
\end{center}
\end{figure}

\begin{figure}
\begin{center}
\includegraphics[height=3in,width=3in]{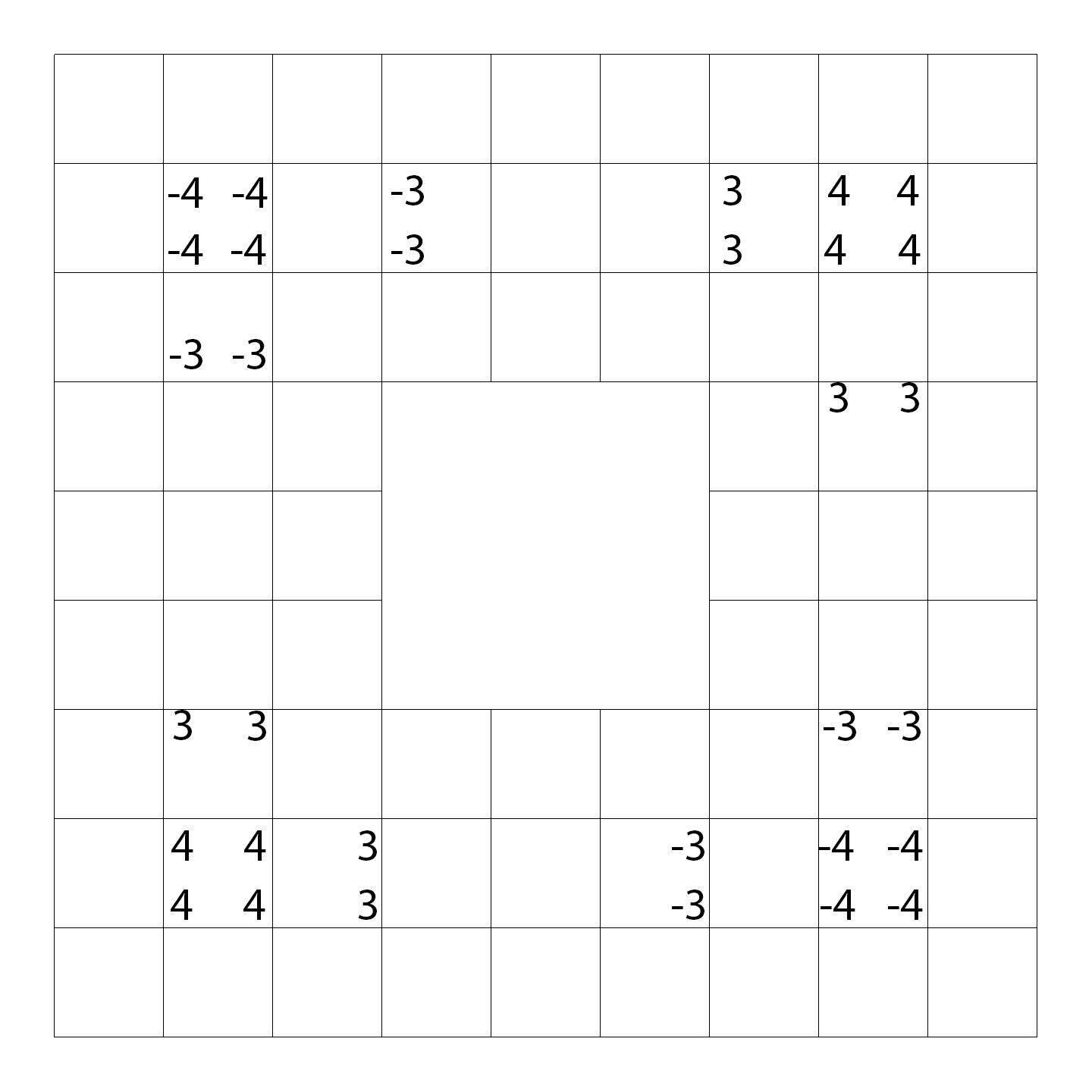}
\caption*{FIGURE 5.6:$\lambda=5$}
\end{center}
\end{figure}

With respect to the dihedral-4 symmetry group of MC, we have a single two-dimensional representations denoted 2 and four one-dimensional representation $ 1\pm \pm$, where the first $\pm$ denotes symmetry or skew-symmetry with respect to the diagonal reflections, and the second $\pm$ denotes symmetry or skew-symmetry with respect to horizontal and vertical reflections. With this notation, the eigenspace corresponding to $\lambda =15$ corresponds to representation 2, the eigenspace $\lambda =9$ corresponds to representation $1 - +$, and the eigenspaces $\lambda = \frac{29\pm\sqrt{75}}{2}$ correspond to the representation $1 + +$.

We have the same miniaturization in passing from $-\Delta_{m-1}$ to $-\Delta_m$ as in the case of SC as described in \cite{BHS}. The exceptions are $\lambda =9$, where the multiplicity is $\frac{2\cdot8^{m-1}+5}{7}$, and $\lambda=15$, where the multiplicity is $\frac{9\cdot8^{m-1}+5}{7}$. In Figure 5.4 we show a basis for the $\lambda=9$ eigenspace with $m=2$, consisting of one $1 - -$ representation and two $1 - +$ representations. Note that the third basis element is the one given by miniaturization. There are two other simple eigenvalues, $\lambda=5$ and $\lambda=12$ that appear with multiplicity one for $m \ge 2$. We show the eigenfunctions for $m=2$ in Figures 5.5 and 5.6, both $ 1 + -$ representations.





\begin{table}[h!]
\centering

\vspace{5mm}

\begin{tabular}{  c c c c }
$\text{Eigenfunction \# 2} $ &   $\text{Eigenfunction \# 3 \& 4}$ \\
\includegraphics[height=1.25in,width=1.75in]{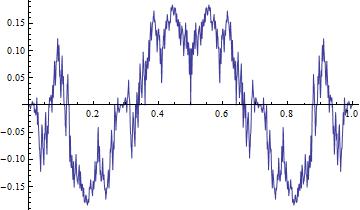} &
 \includegraphics[height=1.25in,width=1.75in]{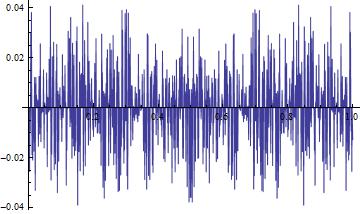} \includegraphics[height=1.25in,width=1.75in]{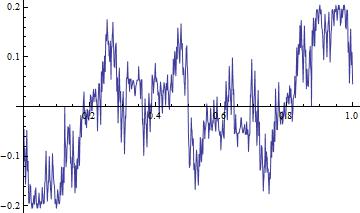}\\
\\
$ \text{Eigenfunction \# 5}  $ & $ \text{Eigenfunction \# 6} $ \\
\includegraphics[height=1.25in,width=1.75in]{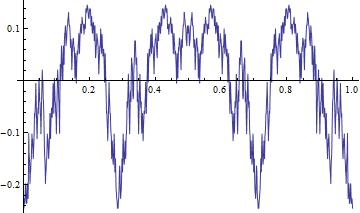}&
\includegraphics[height=1.25in,width=1.75in]{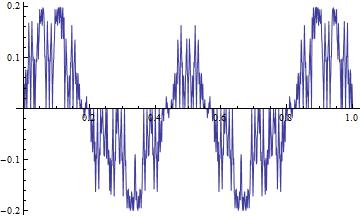}\\
\\ 

$\text{Eigenfunction \# 18}$ & $\text{Eigenfunction \#27 \& 28}$\\
\includegraphics[height=1.25in, width=1.75in]{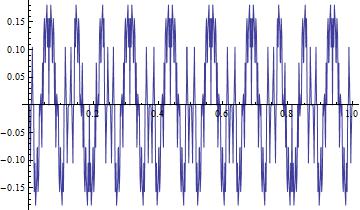}&  \includegraphics[height=1.25in,width=1.75in]{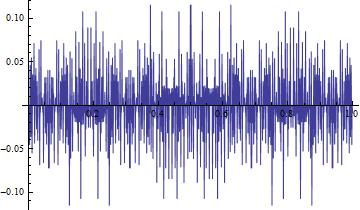} \includegraphics[height=1.25in, width=1.75in]{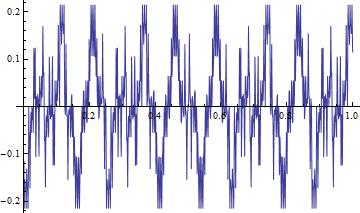}\\
\\
$\text{Eigenfunction \# 29}$ & $\text{Eigenfunction \# 54} $ \\
 \includegraphics[height=1.25in,width=1.75in]{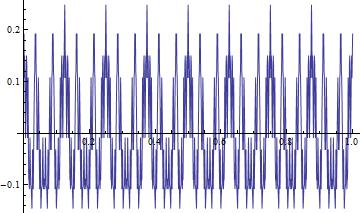} & \includegraphics[height=1.25in, width=1.75in]{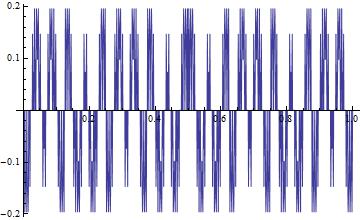}\\
\\

\end{tabular}
\caption*{FIGURE 5.7. Eigenfunctions of the Magic Carpet at Level 3}
\end{table}

\begin{table}
\centering
\begin{adjustwidth}{-1.0in}{-1.0in}
\vspace{3mm}
\begin{tabular}{  c c c c }
$\text{Eigenvalue Counting Function}$\\
\\
$\text{ Level 2}$ & $\text{ Level 3}$ & $\text{Level 4}$ \\
\\
 \includegraphics[height=1.25in,width=1.75in]{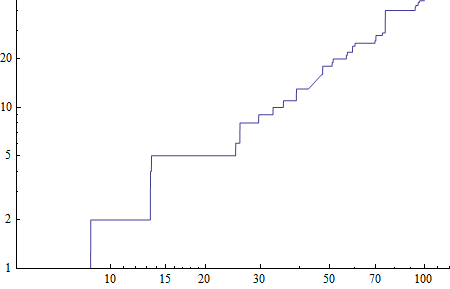}&  \includegraphics[height=1.25in,width=1.75in]{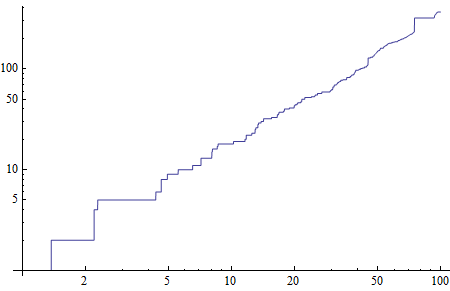} &  \includegraphics[height=1.25in,width=1.75in]{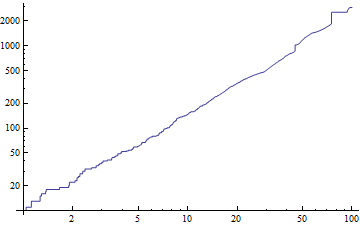}\\ 
\\
$\text{Weyl Ratios}$ \\
\\
$ \text{Level 2}$ & $ \text{Level 3}$ & $\text{Level 4} $\\
\\
 \includegraphics[height=1.25in,width=1.75in]{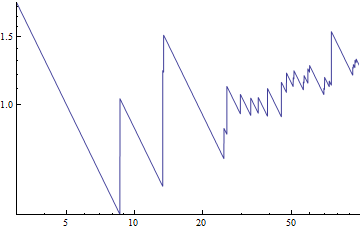} &  \includegraphics[height=1.25in,width=1.75in]{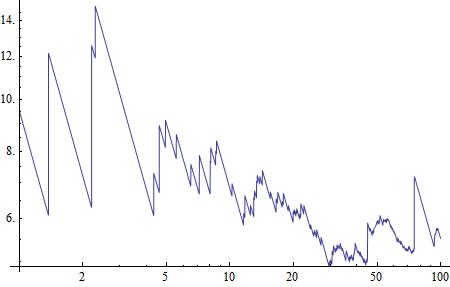} &  \includegraphics[height=1.25in,width=1.75in]{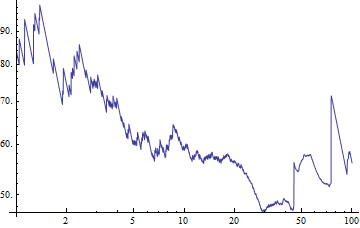}\\
\end{tabular}
\vspace{1mm}

\caption*{FIGURE 5.8. Eigenvalue Counting Function and Weyl Ratio of the Magic Carpet}
\caption*{$\beta=1.2$}
\end{adjustwidth}
\end{table}

\begin{table}[!ht]
\begin{adjustwidth}{-1.3in}{-1.3in}
\begin{center}
$\begin{array}{|c|c|c|c|c|c|c|c|c|c|c|c|||c|c|c|}
\hline
\multicolumn{3}{|c|}{\text {Level 1}}
& \multicolumn{3}{|c|}{\text{Level 2}}
& \multicolumn{3}{|c|}{\text{Level 3}}
& \multicolumn{3}{|c|}{\text{Level 4}}
&\multicolumn{2}{|c|}{\text{Ratio}} \\
\hline
\#  & \text{Mult} & \text{Eiv} &  \# & \text{Mult} & \text{Eiv}& \# & \text{Mult} & \text{Eiv}& \# & \text{Mult} & \text{Eiv}& \frac{\lambda_2}{\lambda_3}&\frac{\lambda_3}{\lambda_4}\\
\hline 1 & 1 & 0 & 1 & 1 & 0 & 1&1&0&1&1&0&&\\  \hline
\hline 2 &1&9.000& 2 & 1&1.726& 2&1& 0.274&2&1&0.0429&6.281&6.406\\ \hline
\hline 3&1&10.228&3&2&2.674&3&2&0.441&3&2&0.068&6.069&6.442\\ \hline
\hline 4&2& 15.000&5&1&2.697&5&1&0.458&5&1&0.072&5.885&6.365\\ \hline
\hline 6&1&18.772&6&1&5.000&6&1&0.869&6&1&0.138&5.752&6.294\\ \hline
\hline &&&7&2&5.515&7&2&0.923&7&2&0.146&5.586&6.296\\ \hline
\hline &&&9&1&5.917&9&1&0.987&9&1&0.154&5.993&6.402\\ \hline
\hline &&&10&1&6.580&10&1&1.112&10&1&0.173&5.915&6.423\\ \hline
\hline &&&11&1&7.102&11&1&1.304&11&1&0.207&5.444&6.290\\ \hline
\hline &&&12&2&7.808&12&2&1.431&12&2&0.223&5.455&6.412\\ \hline
\hline &&&14&3&9.000&14&2&1.610&14&2&0.232&&6.386\\ \hline
\hline &&&17&2&9.475&16&1&1.620&16&1&0.257&&6.305\\ \hline
\hline &&&19&1&10.147&17&1&1.709&17&1&0.272&&6.277\\ \hline
\hline &&&20&1&10.228&18&1&1.726&18&1&0.274&&6.281\\ \hline
\hline &&&21&1&11.261&19&1&2.044&19&1&0.331&&6.168\\ \hline
\hline &&&22&1&11.347&20&1& 2.321&20&1&0.375&&6.177\\ \hline
\hline &&&23&2&11.796&21&2&2.354&21&2&0.379&&6.193\\ \hline
\hline &&&25&1&12.000&23&1&2.501&23&1&0.411&&6.084\\ \hline
\hline &&&26&1&13.893&24&2&2.594&24&2&0.423&&6.131\\ \hline
\hline &&&27& 2&13.998&26&1&2.613&26&1&0.430&&6.068\\ \hline
\hline &&&29&1&14.675&27&2&2.674&27&2&0.440&&6.070\\ \hline
\hline &&&30&11&15.000&29&1&2.697&29&1&0.457&&5.896\\ \hline
\hline &&&41&1&18.663&30&1&2.773&30&1&0.458&&6.051\\ \hline
\hline &&&42&1&18.6701&31&2&2.846&31&2&0.472&&6.021\\ \hline
\hline &&&43&1&18.772&33&1&3.108&33&1&0.518&&5.993\\ \hline
\hline &&&44&2&19.087&34&2&3.314&34&2&0.553&&5.990\\ \hline
\hline &&&46&1&19.316&36&1&3.361&36&1&0.571&&5.882\\ \hline
\end{array}$
\vspace{5mm}
\caption{Eigenvalues of the Magic Carpet}
\end{center}
\end{adjustwidth}
\end{table}

\begin{table}[!ht]
\begin{adjustwidth}{-.8in}{-.5in}
\begin{center}
$\begin{array}{|c|c|c|c|c|c|c|c|c|c|c|c|} 
\hline
\multicolumn{3}{|c|}{\text {Level 1}}
& \multicolumn{3}{|c|}{\text{Level 2}}
& \multicolumn{3}{|c|}{\text{Level 3}}
& \multicolumn{3}{|c|}{\text{Level 4}}\\
\hline
\#  & \text{Mult} & \text{Eiv} &  \# & \text{Mult} & \text{Eiv}& \# & \text{Mult} & \text{Eiv}& \# & \text{Mult} & \text{Eiv}\\
\hline 1 & 1 & 0 & 1 & 1 & 0 & 1&1&0&1&1&0\\  \hline
\hline 2 &1&57.600&2 & 1& 70.701 &2&1&72.041&2&1&71.974\\ \hline
\hline 3&1&65.459&3&2&109.553&3&2&115.511&3&2&114.756\\ \hline
\hline 4&2&96.000&5&1&110.497&5&1&120.146&5&1&120.795\\ \hline
\hline 6&1&120.141&6&1&204.800&6&1&227.868&6&1&231.693\\ \hline
\hline &&&7&2&211.239&7&2&241.978&7&2&245.953\\ \hline
\hline &&&9&1&242.373&9&1&258.799&9&1&258.704\\ \hline
\hline &&&10&1&269.530&10&1&291.629&10&1&290.581\\ \hline
\hline &&&11&1&290.936&11&1&342.0&11&1&347.959\\ \hline
\hline &&&12&2&319.840&12&2&375.214&12&2&374.467\\ \hline
\hline &&&14&3&368.640&14&2&422.201&14&2&423.121\\ \hline
\hline &&&17&2&388.136&16&1&424.783&16&1&431.174\\ \hline
\hline &&&19&1&415.633&17&1&448.0801&17&1&456.843\\ \hline
\hline &&&20&1&418.938&18&1&452.486&18&1&461.037\\ \hline
\hline &&&21&1&461.277&19&1&535.914&19&1&555.996\\ \hline
\hline &&&22&1&464.809&20&1&608.567&20&1&630.487\\ \hline
\hline &&&23&2&483.192&21&2&616.822&21&2&637.366\\ \hline
\hline &&&25&1&491.521&23&1&655.682&23&1&689.711\\ \hline
\hline &&&26&1&569.094&24&2&680.226&24&2&710.011\\ \hline
\hline &&&27&2&573.394&26&1&685.060&26&1&722.426\\ \hline
\hline &&&29&1&601.112&27&2&701.143&27&2&739.204\\ \hline
\hline &&&30&11&614.400&29&1&707.183&29&1&767.555\\ \hline
\hline &&&41&1&764.436&30&1&726.975&30&1&768.899\\ \hline
\hline &&&42&1&764.727&31&2&746.101&32&2&793.059\\ \hline
\hline &&&43&1&768.901&33&1&814.798&33&1&870.066\\ \hline
\hline &&&44&2&781.840&34&2&868.857&34&2&928.283\\ \hline
\hline &&&46&1&791.191&36&1&881.136&36&1&958.650\\ \hline
\end{array}$
\vspace{5mm}
\caption{Renormalized Eigenvalues of the Magic Carpet}
\end{center}
\end{adjustwidth}
\end{table}

In Table 5.1 we show the eigenvalues for levels $m=1,2,3,4$ and their ratios. The ratio values suggest that the eigenvalue renormalization factor should be around 6.4. In Table 5.2 we show the renormalized eigenvalues (multiplied by $(6.4)^m$). Because 6.4 is smaller than the measure renormalization factor 8, this suggests that the energy renormalization factor would have to be around 1.25. Since this is greater than one, it would imply that points have zero capacity and functions of finite energy do not have to be continuous, in contrast to all PCF fractals, SC and PG. We also observe that the spectral data agrees exactly with the data in \cite{kform} for the approximations to the zero-forms Laplacian on MC. In fact the approximate graph Laplacians are identical. 

As in the case of OG, it appears that the only multiplicities in the spectrum of $ -\Delta $ will be one and two, as the higher multiplicities in the spectrum of $-\Delta_m$ occur high up in the spectrum and will not survive in the limit. The only noticeable spectral gap in $-\Delta_m$ occur near $\lambda=15$ and again will not survive in the limit. There may be smaller spectral gaps that survive in the limit, especially since the average separation of eigenvalues goes to zero. This remains to be investigated.
In Figures 5.7 and 5.8 we show the graph of some eigenfunction on the parameter circle. Note that the Peano curve does not respect the symmetries of the MC, so these graphs do not show the kind of symmetry found in the cases of PG or OG. In Figure 5.9 we show graphs of the eigenvalue counting function and the Weyl ratio at levels 2,3, and 4. 

\newpage
\newpage

\section{The Torus and the Triangle}
The Peano curves to the torus, $T_0$, and the triangle, $T_r$, yield graph approximations that are identical to the standard lattice graph approximations. For $T_o$, the graph $\Gamma_m$ has vertices that we may identify with the points $\left(\frac{j}{3^m},\frac{k}{3^m}\right) \mod 1$ with $0\leq j,\ k\leq3^m$. The neighbors are the four points $\left(\frac{j\pm 1}{3^m},\frac{k\pm 1}{3^m}\right)$. Thus the Laplacian $-\Delta_m$ is 
\begin{equation}
-\Delta_m u\left(\frac{j}{3^m},\frac{k}{3^m}\right)=\sum\left(u\left(\frac{j}{3^m},\frac{k}{3^m}\right)-u\left(\frac{j\pm 1}{3^m},\frac{k\pm 1}{3^m}\right)\right).
\end{equation}
Of course the formula looks different in terms of the parameterization $\gamma(t)$ for $t=\frac{n}{2\cdot9^m}$ with identifications. The eigenvalues will be exactly the same, while the eigenfunctions, $e^{2\pi i(px+qy)}$ for $(p,q)\in\mathbb{Z}^2$ (with eigenvalue $4\pi^2(p^2+q^2)$), will have a different appearance as a function of $t$. Aside from the zero-eigenspace, all eigenvalues have multiplicities equal to a multiple of 4 (if $p\neq q$, then multiplicity is at least eight, including $(\pm p,\pm q)$ and $(\pm q, \pm p)$). Because the Peano curve does not respect the dihedral-4 symmetries of the torus, it is difficult to separate out specific eigenfunctions within each eigenspace, so we have been unable to ``interpret" the graphs of eigenfunctions for this example. It is obvious from $(6.1)$ that $9^m\Delta_m\to -\Delta$. In Table 6.1 we show the eigenvalues of $\Delta_m$ and their ratios. In Table 6.2 we show the eigenvalues normalized by multiplication by $\frac{9^m}{4\pi^2}$. The deviation from the expected integer values $p^2+q^2$ is small enough at the low end of the spectrum to confirm the convergence, but it rapidly grows out of hand as the eigenvalues increase. This just confirms that this finite difference method has rather poor accuracy.

\begin{figure}
\begin{center}
\includegraphics[scale=.50,trim = 0mm 0mm 0mm 0mm, clip]{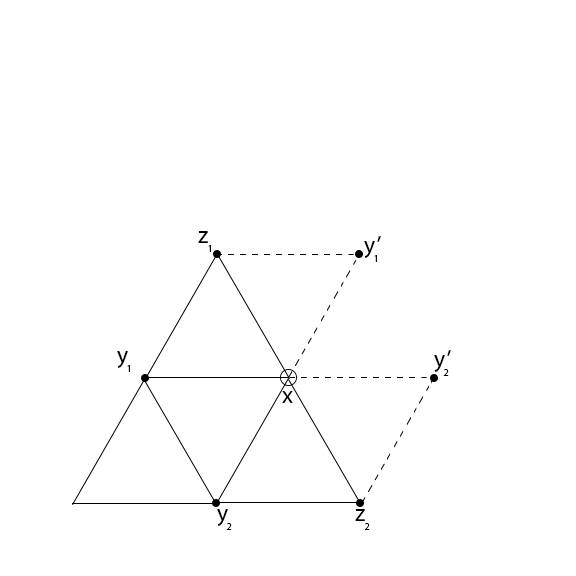}
\caption*{FIGURE 6.1: $x$ is a boundary point, with two boundary neighbors $z_1$,$z_2$ and two interior neighbors $y_1$,$y_2$ in the triangle. We add on two virtual triangles (indicated by dotted lines) with new vertices $y_1'$ and $y_2'$, and even reflection makes $u(y_1')=u(y_1)$ and $u(y_2')=u(y_2)$}
\end{center}
\end{figure}

\begin{table}
\begin{center}
$\begin{array}{|c|c|c|c|c|c|}
\hline
\multicolumn{3}{|c|}{\text{Level 1}}
& \multicolumn{3}{|c|}{\text{Level 2}}\\
\hline
\# & \text{Mult} & \text{Eiv} & \# & \text{Mult} & \text{Eiv}\\ \hline
1 & 1 & 0 & 1 & 1 & 0\\ \hline
2 & 4 & 3 & 2 & 4 & 0.4679\\ \hline
6 & 4 & 6 & 6 & 4 & 0.9358\\ \hline
&&& 10 & 4 & 1.6527\\ \hline
&&&14 & 8 & 2.1206\\ \hline
&&&22 & 7 & 3\\ \hline
&&&29 &  4 & 3.3051\\ \hline
&&&33 &  8 & 3.4679\\ \hline
&&&41 & 4 & 3.8793\\ \hline
&&&45 & 8 & 4.3473\\ \hline
&&&53 & 8 &4.6527\\ \hline
\end{array}$
\vspace{1mm}

\caption*{TABLE 6.1: Eigenvalues of the Torus}
\end{center}
\end{table}

\begin{table}
\begin{center}
$\begin{array}{|c|c|c|c|c|c|}
\hline
\multicolumn{3}{|c|}{\text{Level 1}}
& \multicolumn{3}{|c|}{\text{Level 2}}\\
\hline
\# & \text{Mult} & \text{Eiv} & \# & \text{Mult} & \text{Eiv}\\ \hline
1 & 1 & 0 & 1 & 1 & 0\\ \hline
2 & 4 & 1 & 2 & 4 & 1\\ \hline
6 & 4 & 2 & 6 & 4 & 2\\ \hline
&&& 10 & 4 & 3.5320\\ \hline
&&&14 & 8 & 4.5320\\ \hline
&&&22 & 7 & 6.4114\\ \hline
&&&29 &  4 & 7.0641\\ \hline
&&&33 &  8 & 7.4114\\ \hline
&&&41 & 4 & 8.2908\\ \hline
&&&45 & 8 & 9.2908\\ \hline
&&&53 & 8 &9.9435\\ \hline
\end{array}$
\vspace{1mm}

\caption*{TABLE 6.2: Renormalized eigenvalues of the Torus}
\end{center}
\end{table}

The situation for the triangle is much better. The Peano curve again produces graph approximations which are identical to the triangular lattice graphs. At level $m$ we subdivide each side of the triangle into $2^m$ equal segments, and the intersection points of the lines joining the vertices on the sides are the vertices of the graph. Interior points have four neighbors (the two interior neighbors carry twice the conductance of the boundary neighbors), and the three corners of the triangle have two neighbors. The approximate Laplacian may be written

\begin{equation}
-\Delta_m u(x)=6 u(x)-
\begin{cases} 
\sum_{y\underset{m} \sim x} u(y) &  $\text{$x$ an interior point}$ \\
2\sum_{y\underset{m}\sim x} u(y)+\sum_{z\underset{m} \sim x} u(z) &\text{$x$ a boundary point}\\ & \text{$z$ its boundary neighbors}\\
\\
3\sum_{z\underset{m}\sim x} u(z) & \text{$x$ a corner point.}
\end{cases}
\end{equation}

\begin{table}
\begin{adjustwidth}{-1.3in}{-1in}
\begin{center}
$\begin{array}{|c|c|c|c|c|c|c|c|c|c|c|c|c||c|c|c|} 
\hline
\multicolumn{3}{|c|}{\text {Level 1}}
& \multicolumn{3}{|c|}{\text{Level 2}}
& \multicolumn{3}{|c|}{\text{Level 3}}
& \multicolumn{4}{|c|}{\text{Level 4}}
& \multicolumn{3}{|c|}{\text{Ratios}}\\
\hline
\#  & \text{Mult} & \text{Eiv} &  \# & \text{Mult} & \text{Eiv}& \# & \text{Mult} & \text{Eiv}& \# & \text{Mult} & \text{Eiv}& \text{Norm. Eiv} & \frac{\lambda_1}{\lambda_2}&\frac{\lambda_2}{\lambda_3} &\frac{\lambda_3}{\lambda_4} \\
\hline 1 & 1 & 0 & 1 & 1 & 0 & 1&1&0&1&1&0&&&\\  \hline
\hline 2 &2&0.5120 	&2 & 2& 0.1347 &2&2&0.0341 &2&2&0.0086&1&3.8&3.95&3.97 \\ \hline
\hline 4&1&1.3333&4&1&0.3905&4&1&0.1015 &4&1&0.0256&2.976&3.41&3.85&3.96 \\ \hline
\hline 5&2&1.6667 &5&2&0.5120 &5&2&0.1347&5&2&0.0341&3.965&3.26&3.8&3.95\\ \hline
\hline 7&2&2.3333 &7&2&0.8502 &7&2&0.2328 &7&2&0.0595&6.918&2.74&3.65&3.91 \\ \hline
\hline 9&3&2.6667 &9&2&1.0572 &9&2&0.2968 &9&2&0.0764&8.883&2.52&3.56&3.88 \\ \hline
\hline 12&2&2.8214 &11&1&1.3333 &11&1&0.3905 &11&1&0.1015&11.802&2.12&3.41&3.85\\ \hline
\hline 14&2&3.0000 &12&2&1.4227 &12&2&0.4213&12&2&0.1098&12.767&2.11&3.38&3.84 \\ \hline
\hline &&&14&2&1.6667 &14&2&0.5120 &14&2&0.1348&15.674&&3.26&3.8 \\ \hline
\hline &&&16&2&1.8619 &16&2&0.5999 &16&2&0.1595&18.456&&3.1&3.76 \\ \hline
\hline &&&18&2&2.0000 &18&2&0.6576&18&2&0.1759&20.453&&3.04&3.74\\ \hline
\hline &&&20&2&2.2323 &20&2&0.7697 &20&2&0.2085 &24.244&&2.9&3.69\\ \hline
\hline &&&22&1& 2.2761 &22&1&0.8231 &22&1&0.2247 &26.127&&2.77&3.66	\\ \hline
\hline &&&23&2&2.3333 &23&2&0.8502 &23&2&0.2328&27.069&&2.74&3.65 \\ \hline
\hline &&&25&2&2.4832&25&2&0.9298 &25&2&0.2569&29.872&&2.67&3.62 \\ \hline

\end{array}$
\vspace{5mm}
\caption*{TABLE 6.3: Eigenvalues of the Triangle}
\end{center}
\end{adjustwidth}
\end{table}

Note that this is an approximation to the Neumann Laplacian on $T_r$, since even reflection across a boundary line to a virtual neighboring triangle transforms 
\[2\sum_{y\underset{m}\sim x} u(y)+\sum_{z\underset{m} \sim x} u(z)\]
into the sum of $u$ at the six neighboring vertices in the larger configuration(Figure 6.1).
Then we have 
\begin{equation}-4^m\Delta_m\to \frac{3}{2}\Delta \text{(Neumann boundary conditions) }\end{equation}
(the factor $\frac{3}{2}$ comes from $\left(\frac{\partial}{\partial  x}\right)^2+\left(\frac{1}{2}\frac{\partial}{\partial x}+\frac{\sqrt{3}}{2}\frac{\partial}{\partial y}\right)^2+\left(\frac{1}{2}\frac{\partial}{\partial x}-\frac{\sqrt{3}}{2}\frac{\partial}{\partial y}\right)^2=\frac{3}{2}\Delta$). In Table 6.3 we show the eigenvalues at different levels and their ratios, and on level 4 the eigenvalues normalized by multiplication by $4^m\frac{2}{3}\left(\frac{3}{4\pi}\right)^2$, to be compared with the integer values $p^2+q^2+pq$.
However, now we are able to make sense of the graphs of the eigenfunctions as a function of the circle parameter $t$. In fact we will argue that this alternate way of visualizing eigenfunctions offers some appealing advantages to the rather awkward view of functions defined on $T_r$. The dihedral-3 symmetry group acting on $T_r$ is not completely respected by the Peano curve, but the subgroup of rotations $\left(\text{through angles }0,\ \frac{2\pi}{3}, \ \frac{4\pi}{3}\right)$ is. A rotation through the angle $\frac{2\pi}{3}$ amounts to the translation $t\to t+\frac{1}{3}$, so any function invariant under the rotation subgroup is represented by a function periodic of period $\frac{1}{3}$, a property instantly visible from the graph.

There is also a different symmetry, not part of the dihedral-3 symmetry group, that plays an important role in miniaturization. Take any Neumann eigenfunction $u$ on $T_r$ with eigenvalue $\lambda_j$, shrink it to the subtriangle $\frac{1}{2}\ T_r$ by dilation, and reflect it in each of the interior sides to the remaining subtriangles. Because of the Neumann boundary conditions this miniaturization produces an eigenfunction $\tilde u$ with eigenvalue $4\lambda$. In the circle parameterization $\tilde u (t)=u(4t)$, so we obtain a function that is periodic of period $\frac{1}{4}$. Of course if the initial $u$ was rotation invariant, then the period of $\tilde u$ is $\frac{1}{12}$. By iterating miniaturization we may obtain functions that have period $\frac{1}{4^n}$ or $\frac{1}{3\cdot 4^n}$. All these periods are immediately apparent from the graphs, shown in Figure 6.2.

But we can say much more precisely where these periods occur. The Neumann spectrum of $T_r$ is well-known.  Suppose the triangle has side length 1 and corners at $(0,0),\ \left(\frac{\sqrt{3}}{2},\frac{1}{2}\right)$ and $\left(\frac{\sqrt{3}}{2},\frac{-1}{2}\right)$. For every pair $(p,q)$ of non-negative integers define
\begin{equation}
u(p,q)=\quad\quad\quad\quad\quad\quad\quad\quad\quad\quad\quad\quad
\quad\quad\quad\quad\quad\quad\quad\quad\quad\quad\quad\quad\quad\quad
\quad\quad\quad
\end{equation}
\quad$e(p,q)+e(-p,-q)+e(-p,p+q)+e(q,-p-q)+e(-p-q,p)+e(p+q,-q),$\\
\\
where
\begin{equation}
u(p,q)(x)=e^{2\pi i(pv+qw)\cdot x}\text{ for }v=\left(\frac{1}{\sqrt{3}},\frac{1}{3}\right),\ w=\left(0,\frac{2}{3}\right).
\end{equation}
Then $u(p,q)$ is a Neumann eigenfunction with eigenvalue $\left(\frac{4\pi}{3}\right)^2(p^2+q^2+pq)$, and these are the only ones. Note that when $p=q$ we obtain an eigenspace of multiplicity one, and for $p\neq q$ the functions $u(p,q)$ and $u(q,p)$ span an eigenspace of multiplicity two. Coincidences where $p^2+q^2+pq=(p')^2+(q')^2+p'q'$ may lead to higher multiplicities, but this does not change the narrative substantially.

\begin{table}
\begin{adjustwidth}{-1.3in}{-1in}
\begin{center}

\vspace{5mm}

\begin{tabular}{  c c c c }
$\text{Eigenfunction \# 2 \& 3} $ &   $\text{Eigenfunction \# 4}$ \\
\includegraphics[height=1.25in,width=1.75in]{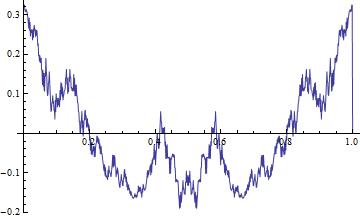} \includegraphics[height=1.25in,width=1.75in]{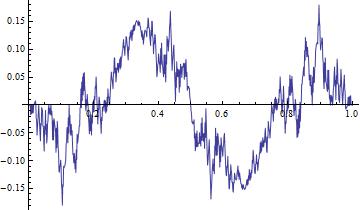} &
 \includegraphics[height=1.25in,width=1.75in]{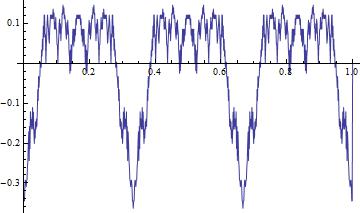} \\
\\
$ \text{Eigenfunction \# 5 \& 6}  $ & $ \text{Eigenfunction \# 7 \& 8} $ \\
\includegraphics[height=1.25in,width=1.75in]{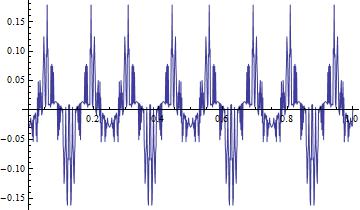} \includegraphics[height=1.25in,width=1.75in]{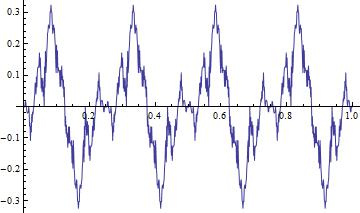}&
\includegraphics[height=1.25in,width=1.75in]{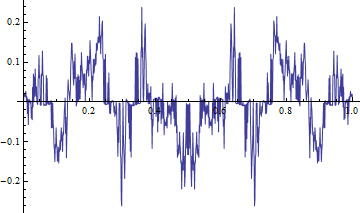} \includegraphics[height=1.25in, width=1.75in]{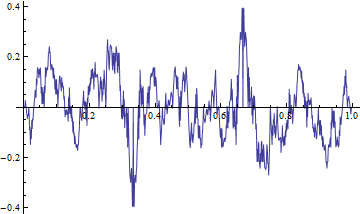}\\
\\ 

$\text{Eigenfunction \# 9 \& 10}$ & $\text{Eigenfunction \#11}$\\
\includegraphics[height=1.25in, width=1.75in]{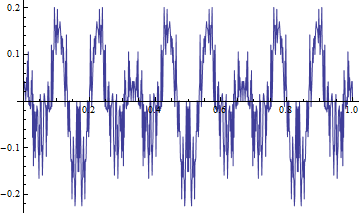} \includegraphics[height=1.25in, width=1.75in]{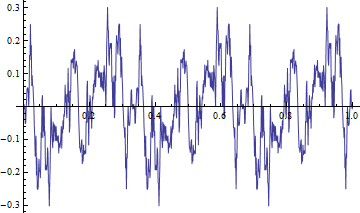}&  \includegraphics[height=1.25in,width=1.75in]{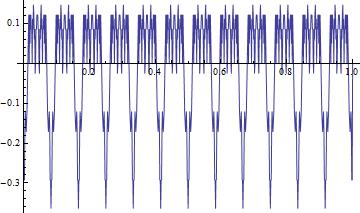} 
\\
$\text{Eigenfunction \# 14 \& 15}$ \\
 \includegraphics[height=1.25in,width=1.75in]{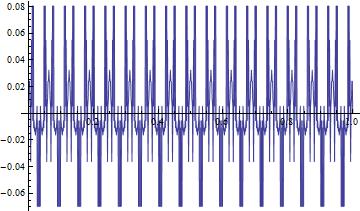}  \includegraphics[height=1.25in, width=1.75in]{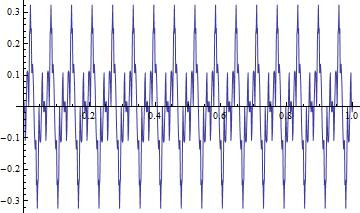}\\
\\

\end{tabular}
\end{center}
\caption*{FIGURE 6.2: Eigenfunctions of the Triangle at Level 4}
\end{adjustwidth}

\end{table}

The $p=q$ multiplicity one space transforms according to the $1+$ representation (symmetric with respect to reflections) and the formula simplifies to $u(p,p)=2\cos 2\pi p\frac{\sqrt{3}}{3}x\cos 2\pi p y+\cos 2\pi p\frac{2\sqrt{3}}{3}x$, and so $u(2^n(2^l+1),2^n(2^l+1))$ also transforms according to the $1+$ representation and has period $\frac{1}{3\cdot 4^n}$. This is seen in $\# 4$ corresponding to$(1,1)$ and $\# 11$ $(2,2)$. When $p\neq q$ and both are odd then the multiplicity two space transforms according to the 2 representation and has no periodicity. When $p$ is even and $q$ is odd, then there are two cases: unless $q=p\pm 3$, it is again the 2 representation, seen in $\#7$ $\&$ $8$ $(2,1)$,but if $q=p\pm 3$ then the space breaks up into a direct sum of a $1+$ and a $1-$ representation, and both have period $\frac{1}{3}$, seen in $\# 9 \& 10$ $(0,3)$. In fact, taking the sum and difference of (6.4) for $(p+3,p)$ and $(p,p+3)$ yields the formulas 
\begin{equation}
\cos 2\pi \frac{p+3}{\sqrt{3}} x\cos 2\pi(p+1)y+\cos 2\pi \frac{p}{\sqrt{3}} x\cos 2\pi(p+2)y+\cos 2\pi \frac{2p+3}{\sqrt{3}} x\cos 2\pi y
\end{equation}
for the $1+$ function and 
\begin{equation}
\cos 2\pi \frac{p+3}{\sqrt{3}} x\sin 2\pi(p+1)y+\cos 2\pi \frac{p}{\sqrt{3}} x\sin 2\pi(p+2)y+\cos 2\pi \frac{2p+3}{\sqrt{3}} x\sin 2\pi y
\end{equation}
for the $1-$ function. Finally, if $p$ and $q$ are both even with $(p,q)=2^n(p',q')$ with at least one of $p',q'$ odd, then the space is the $n$-fold iterated miniaturization of the $(p',q')$ space, with the same representations and the period multiplied by $\frac{1}{4^n}$. These behaviors are seen in $\#$ $5$ $\&$ $6$ $(2,0)$ and $\#$ $14$ $\&$ $15$ $(4,0)$.

\bigskip

\emph{*Any opinions, findings, and conclusions or recommendations expressed in this material are those of the authors and do not necessarily reflect the views of the National Science Foundation}


\begin{thebibliography}{9}

\bibitem{Spec} B.Adams, S.A. Smith, R.S. Strichartz and A. Teplyaev,
\emph{The Spectrum of the Laplacian on the Pentagasket},Trends in Mathematics: Fractals in Graz 2001,\textbf{1-24},(2002), Birkhauser, Basel. 

\bibitem{famofjulia} T.Aougab, S.C. Dong and R.S. Strichartz, \emph{Laplacians on a family of quadratics Julia sets II}, Commun. Pure Appl. Anal., 12, 2013, \textbf{1-58}

\bibitem{kform} J. Bello, Y. Li and R.S. Strichartz, \emph{Hodge-deRham theory of K-forms on carpet type fractals}, In preperation. 


\bibitem{BHS} T.Berry, S.Heilman and R.S. Strichartz,{\em Outer Approximation of the Spectrum of a Fractal Laplacian}, Experimental Mathematics 18:4 (2000) pp.\ 449-480.

\bibitem{julia} Taryn\ Flock and Robert\ Strichartz. {\em Laplacians on a Family of Quadratic Julia Sets}, Trans. of the Amer. Math. Soc. 2012 Aug;364(8):3915-3965.


\bibitem{ki} Jun\ Kigami. {\em Analysis on Fractals}, Cambridge University Press 2008.


\bibitem{JuliaSetsIII} C. Spicer, R.S. Strichartz, and E. Totari, \emph{Laplacians on Julia Sets III: cubic Julia sets and formal matings}, Contemporary Math, 600 (2013), \textbf{327-348}

\bibitem{FourierSeries} R.S. Strichartz, \emph{ Laplacians on fractals with spectral gaps have nicer Fourier series}, Math. Res. Lett. 12(2005), \textbf{269-274}. 

\bibitem{str} Robert\ Strichartz. {\em Differential Equations on Fractals: A Tutorial}, Princeton University Press. 2006.


\end{thebibliography}
\end{document}